\documentclass[a4paper,11pt,reqno]{amsart}

\usepackage{amsmath}
\usepackage{bbm}
\usepackage{amssymb}
\usepackage{amsthm}
\usepackage{amscd}
\usepackage{bbold}

\usepackage{hyperref}

\usepackage[top=3cm,left=3cm,right=3cm,bottom=3.5cm]{geometry}

\usepackage[english]{babel}
\usepackage[utf8]{inputenc} 
\usepackage[T1]{fontenc}
\usepackage{ esint }

\usepackage{authblk}

\usepackage[T1]{fontenc}
\usepackage{lmodern}

\usepackage[colorinlistoftodos,prependcaption]{todonotes}

\newcommand{\Addresses}{{
		\bigskip
		\footnotesize
		
	 \textsc{MIT, Dept. of Math., 77 Massachusetts Avenue, Cambridge, MA 02139-4307.}\par\nopagebreak
		\textit{E-mail address}, \texttt{ozuch@mit.edu}
	}}

\newtheorem{thm}{Theorem}[section]

\newtheorem{lem}[thm]{Lemma}
\newtheorem{prop}[thm]{Proposition}
\newtheorem{cor}[thm]{Corollary}

\newtheorem{defn}[thm]{Definition}
\newtheorem{conj}[thm]{Conjecture}
\newtheorem{exmp}[thm]{Example}
\newtheorem{quest}[thm]{Question}

\newtheorem{rem}[thm]{Remark}
\newtheorem{note}[thm]{Note}

\DeclareMathOperator{\Hess}{\operatorname{Hess}}
\DeclareMathOperator{\Rm}{\operatorname{Rm}}
\DeclareMathOperator{\R}{\operatorname{R}}
\DeclareMathOperator{\vol}{\operatorname{Vol}}
\DeclareMathOperator{\Ric}{\operatorname{Ric}}

\DeclareMathOperator{\tr}{\operatorname{tr}}
\DeclareMathOperator{\E}{\operatorname{E}}

\DeclareMathOperator{\Id}{\operatorname{Id}}

\author{Tristan OZUCH}
\affil{Massachusetts Institute of Technology}
\date{}

\title{Integrability of Einstein deformations and desingularizations}

\begin{document}

\maketitle
\vspace{-20pt}
\begin{center}
    \Large{Tristan Ozuch}\\
    \normalsize{Massachusetts Institute of Technology}
\end{center}

\begin{abstract}
    We study the question of the integrability of Einstein deformations and relate it to the question of the desingularization of Einstein metrics. Our main application is a negative answer to the long-standing question of whether or not every Einstein $4$-orbifold (which is an Einstein metric space in a synthetic sense) is limit of smooth Einstein $4$-manifolds. We more precisely show that spherical and hyperbolic $4$-orbifolds with the simplest singularities cannot be Gromov-Hausdorff limits of smooth Einstein $4$-metrics without relying on previous integrability assumptions. For this, we analyze the integrability of deformations of Ricci-flat ALE metrics through variations of Schoen's Pohozaev identity. Inspired by Taub's preserved quantity in General Relativity, we also introduce preserved integral quantities based on the symmetries of Einstein metrics. These quantities are obstructions to the integrability of infinitesimal Einstein deformations ``closing up'' inside a hypersurface -- even with change of topology. We show that many previously identified obstructions to the desingularization of Einstein $4$-metrics are equivalent to these quantities on Ricci-flat cones. In particular, all of the obstructions to desingularizations bubbling out Eguchi-Hanson metrics are recovered. This lets us further interpret the obstructions to the desingularization of Einstein metrics as a defect of integrability.
\end{abstract}
    
    \tableofcontents

\section*{Introduction}

An Einstein metric $\mathbf{g}$ satisfies, for some real number $\Lambda$, the equation
\begin{equation}
    \Ric(\mathbf{g})=\Lambda \mathbf{g}.\label{einstein}
\end{equation}
In dimension $4$, these metrics are considered optimal due to the homogeneity of their Ricci curvature but also as critical points of the Einstein-Hilbert functional with fixed volume, $g\mapsto\int_M \R_g dv_g$, and more importantly as minimizers of the $L^2$-norm of the Riemann curvature tensor, $g\mapsto \int_M |\Rm_g|^2dv_g$ often interpreted as an energy. 

From dimension $4$, even under natural assumptions of bounded diameter (compactness) and lower bound on the volume (non-collapsing) Einstein metrics can develop singularities. This first issue means that the set of unit-volume Einstein metrics is not complete for the usual Gromov-Hausdorff distance. Moreover, it has been proven that some infinitesimal Einstein deformations may not integrate into curves of actual Einstein metrics in dimension higher than $4$. This second issue shows that the moduli space of Einstein metrics itself might have singularities. 
\\

In this article, we exhibit links between the two seemingly unrelated above issues and apply the resulting analogy to the question of the desingularization of Einstein $4$-manifolds. 

\subsection*{Completion of the moduli space of Einstein $4$-manifolds}

One major goal for $4$-dimensional geometry is to understand the structure of the compactification of the moduli space of Einstein metrics on a differentiable manifold $M^4$ which is defined as
\begin{equation}
    \mathbf{E}(M^4) := \left\{(M^4,\mathbf{g})\;|\;\exists \Lambda\in \mathbb{R},\; \Ric(\mathbf{g})=\Lambda \mathbf{g},\; \vol(M^4,\mathbf{g})= 1\right\}\slash\mathcal{D}(M^4),\label{def moduli space}
\end{equation}
where $\mathcal{D}(M^4)$ is the group of diffeomorphisms of $M^4$ acting on metrics by pull-back. This space is classically equipped with the Gromov-Hausdorff distance, $d_{GH}$. The metric spaces which are limits of Einstein $4$-manifolds with uniformly controlled diameter and volume, as well as the associated singularity models, have been understood for a long time in the Gromov-Hausdorff sense \cite{and,bkn}: they are respectively \emph{Einstein orbifolds} and \emph{Ricci-flat ALE orbifolds}. The \emph{metric completion} of $(\mathbf{E}(M^4),d_{GH})$ is 
\begin{equation}
    \mathbf{E}(M^4)\;\cup\;\partial_o\mathbf{E}(M^4),\label{def frontière}
\end{equation}
where $\partial_o\mathbf{E}(M^4)$ is the set of $d_{GH}$-limits with bounded diameter (i.e. at finite $d_{GH}$-distance) of Einstein metrics on $M^4$. 
\\

The elements of $\partial_o\mathbf{E}(M^4)$ are Einstein orbifolds and a question which we answer here is the converse. Quoting Anderson \cite{andsurv}: ``It has long been an open question whether Einstein orbifold metrics can be resolved to smooth Einstein metrics close to them in the Gromov-Hausdorff topology.'' We will prove that this resolution is not possible for some of the simplest Einstein orbifolds: the spherical an hyperbolic ones. These orbifolds should therefore not really be considered as \emph{singular} Einstein metrics. In order to prove this, we use the analogy between such a resolution and an Einstein deformation of a flat cone, and we study the potential nonintegrability of Ricci-flat ALE deformations.

\subsection*{Preserved quantities and integrability} 
Let $\mathbf{g}$ be an Einstein manifold on an open subset $\mathcal{U}$ of a manifold. We say that $h$ is an \emph{infinitesimal Einstein deformation} of $\mathbf{g}$ on $\mathcal{U}$ if the perturbation $t\mapsto\mathbf{g}+th$ satisfies \eqref{einstein} at the infinitesimal level as $t\to 0$. We say that $h$ is an \emph{integrable Einstein deformation} if on any compact in $\mathcal{U}$, there exists a smooth curve $t\in[0,1]\mapsto \mathbf{g}_t$ such that \begin{itemize}
    \item $\mathbf{g}_0 = \mathbf{g}$,
    \item $\partial_t{\mathbf{g}_t}_{|t=0} = h$, and
    \item for any $t$, $\mathbf{g}_t$ satisfies \eqref{einstein} with constant $\Lambda(t)\in\mathbb{R}$.
\end{itemize}
It is clear that an integrable Einstein deformation is an infinitesimal Einstein deformation, but the converse is a very delicate question and is not always true. Counter-examples were found in \cite{koi2} in dimension strictly higher than $4$.

The question of integrability of Einstein deformations is crucial in understanding the structure of the moduli space of Einstein metrics, and a major question is whether or not such a moduli space can be really singular, see \cite[12.10]{bes}. It is also crucial for the behavior of Ricci flow near Einstein metrics: in the compact situation, having integrable deformations ensures optimal rates of convergence or divergence, see for instance \cite{has,hm}. In the noncompact situation, this moreover seems to be a necessary condition for the dynamical stability of Ricci-flat ALE metrics, see \cite{kp,do,do2}. 
\\

In the Lorentzian context of General Relativity, it has been proven that the question of integrability of Einstein deformations was completely ruled by the presence of symmetries.

Let us denote $\E^{(2)}_\mathbf{g}(h,h)$ the quadratic terms terms in the development of the Einstein tensor $h\mapsto \E(\mathbf{g}+h):= \Ric(\mathbf{g}+h) - \frac{\R(\mathbf{g}+h)}{2}(\mathbf{g}+h)$. For any Killing vector field $X$, a hypersurface $\Sigma$ with a unit normal $n_\Sigma$ and any infinitesimal Einstein deformation $h$, we define \emph{Taub's preserved quantity}:
\begin{equation}
    \mathcal{T}^\Sigma_{X}(h,h):=\int_\Sigma \big(\E^{(2)}_\mathbf{g}(h,h)\big)(X,n_\Sigma)dv_\Sigma.\label{eq def taub intro}
\end{equation}
The link with the integrability of $h$ is that if $h$ is an integrable \emph{Ricci-flat} deformation, then $\mathcal{T}_X^\Sigma(h,h)$ has to vanish. In this sense, Taub's preserved quantities for all Killing vector fields $X$ are \emph{obstructions} to the integrability of $h$.

A remarkable and beautiful result is that in the Lorentzian context, the above obstructions are the \emph{only} obstructions to the integrability of $h$, see \cite{fmm,amm}. Namely, if they vanish, then, under mild assumptions, one can construct a smooth curve of Ricci-flat metrics starting at $\mathbf{g}$ whose first jet is $h$. The hyperbolic nature of the Einstein equations in the Lorentzian context is an important aspect of the proof.

\subsection*{Einstein deformations of Ricci-flat cones}

Let $(C(S),\mathbf{g}_{C(S)})=(\mathbb{R}^+\times S, dr^2 + r^2 \mathbf{g}_S)$ be a $d$-dimensional Ricci-flat cone with link $(S,\mathbf{g}_S)$ satisfying $\Ric(\mathbf{g}_S)=(d-2)\mathbf{g}_S$. Such cones model the asymptotics of Einstein metrics at given points or at infinity. We will mostly focus on the usual Euclidean cone $(\mathbb{R}^d,\mathbf{e})=(C(\mathbb{S}^{d-1}),\mathbf{g}_{C(\mathbb{S}^{d-1})})$ and its quotients.

We extend the obstructions given by Taub's preserved quantities on the hypersurface $\Sigma = S$ to general Einstein deformations (not necessarily Ricci-flat this time).  The condition rewrites: for any Killing vector field $X$ of $(S,\mathbf{g}_S)$, if $h$ is integrable, then one has
\begin{equation}
    \mathcal{T}^S_{X}(h,h)=\int_S \big(\mathring{\Ric}^{(2)}_{\mathbf{g}_{C(S)}}(h,h)\big)(X,\partial_r)dv_S =0,\label{obst killing intro}
\end{equation}
where $\mathring{\Ric}$ is the traceless part of $\Ric$ and $ \mathring{\Ric}^{(2)} $ its second order variations.

We moreover introduce a similar quantity based on the \emph{conformal} Killing vector field $r\partial_r$ on the cone $C(S)$.
\begin{thm}[Informal, Proposition \ref{seconde obst conforme}]\label{thm intro}
    Let $h$ be an infinitesimal Einstein deformation of the Ricci-flat cone $(C(S),\mathbf{g}_{C(S)})$. Assume that $h$ is either defined and bounded on the interior of $S$ or on its exterior and decaying at infinity. Then, we have the following identity:
    \begin{equation}
        \int_S \big(\mathring{\Ric}^{(2)}_{\mathbf{g}_{C(S)}}(h,h)\big)(r\partial_r,\partial_r)  dv_S+ \textup{other terms} = 0.\label{obst conformal introo}
    \end{equation}
    In all of our situations of interest, we are able to find convenient gauges for $h$ in which these \emph{other terms} vanish thanks to some variations of Schoen's Pohozaev formula from \cite{sch}.
\end{thm}
\begin{rem}
    In the above case, we do not need to assume that $h$ is integrable. This should rather be thought of as an obstruction to the existence of an Einstein metric which \emph{closes up} inside $S$, that is the existence of an Einstein deformation on a compact domain whose only boundary is $S$.
\end{rem}

As is well-known, a difficulty when considering deformations of cones is verifying that the resulting curve of metric is \emph{complete}, see \cite{bl} where this is discussed for hyperkähler deformations of cones. The main challenge is that such deformations typically require changes of topology. When allowing a change of topology, just as in the desingularization of Einstein $4$-manifolds, one needs to consider deformations of $\mathbb{R}^4\slash\Gamma$ with topology $N=(\mathbb{R}^4\slash\Gamma)\backslash\{0\}\cup \Sigma$ for some lower dimensional manifold $\Sigma$. The prototypical example is that of the Eguchi-Hanson metric which can be seen as an Einstein deformation of $\mathbb{R}^4\slash\mathbb{Z}_2$ on $T^*\mathbb{S}^2=(\mathbb{R}^4\slash\mathbb{Z}_2)\backslash\{0\}\cup \mathbb{S}^2$. Here, we may see the Euclidean metric $\mathbf{e}$ on $\mathbb{R}^4\slash\mathbb{Z}_2$ as living on $T^*\mathbb{S}^2$ while being \emph{degenerate} on $\mathbb{S}^2$, that is the restriction of $\mathbf{e}$ to the submanifold $\mathbb{S}^2$ vanishes.

In Section \ref{closing up EH}, we discuss the situation of an Einstein metric coming out of $\mathbb{R}^4\slash\mathbb{Z}_2$ and closing-up inside with topology $T^*\mathbb{S}^2$. Assume that there exists a smooth curve of nondegenerate Einstein metrics $t\in [0,1]\mapsto \mathbf{g}_t$ on $ \big(B_\mathbf{e}(0,1)\backslash\{0\}\big)\cup \mathbb{S}^2 \subset T^*\mathbb{S}^2 $ with $\mathbf{g}_0 =\mathbf{e}$ and $ {\partial_t}_{|t=0}\mathbf{g}_t = h$ outside $\mathbb{S}^2$ in well-chosen coordinates. Then we have the obstructions:
    \begin{equation}
        \int_{\mathbb{S}^{3}\slash\mathbb{Z}_2} \big(\mathring{\Ric}_\mathbf{e}^{(2)}(h,h)\big)(r\partial_r,\partial_r)dv_{\mathbb{S}^{3}\slash\mathbb{Z}_2} = 0,\label{obst conformal intro}
    \end{equation}
and the similar ones from Killing vector fields. In that situation, an Eguchi-Hanson metric bubbles-out of the cone $\mathbb{R}^4\slash\mathbb{Z}_2$ and the link with the question of desingularization becomes clear. For some related results about hyperkähler metrics on manifolds with boundaries, see \cite{fls,biqbord}. 
\\

A natural question would be whether one could find similar obstructions on typical cones of dimension higher than $4$ with codimension $4$ singularities.
\begin{quest}
    Can we find similar obstructions for Einstein deformations of cones on singular Einstein orbifolds like $\mathbb{R}^k\times( \mathbb{R}^4\slash\mathbb{Z}_2)$ for $k\geqslant 1$? 
\end{quest}

\subsection*{Vanishing of the obstructions in dimension $4$}

We then test the conditions \eqref{obst killing intro} and \eqref{obst conformal introo} considering simple Einstein deformations coming from two situations: the rescaling of Einstein metrics at a given point, and the rescaling of Ricci-flat ALE metrics at infinity, and show that they vanish for arbitrary deformations.

The vanishing of the obstructions is not surprising in the case of Einstein metrics around a given point: it has been shown that for any curvature satisfying the Einstein condition at the given point, there exists a germ of Einstein metric with the corresponding curvature in \cite{gas}. It implies that the obstructions \eqref{obst killing intro} and \eqref{obst conformal intro} in this situation actually vanish in any dimension. 

These results in Section \ref{section 4d} and the Appendix \ref{appendix mf} can also be seen as consequences of \cite{and08}. Their proofs have the merit to introduce natural systems of coordinates in which the computations of the quadratic terms of Ricci curvature are convenient (and where the \emph{other terms} of \eqref{obst conformal intro} are vanishing), we note that the ALE coordinates considered are based on \cite{bh}. These coordinates are crucial for the next case of the desingularization of an Einstein orbifold, where the obstructions do not vanish.

\subsection*{Obstruction to the desingularization of Einstein $4$-manifolds}

We finally study the degeneration of Einstein $4$-manifolds, that is the $d_{GH}$-convergence of metrics in $\mathbf{E}(M)$ to the boundary $\partial_o\mathbf{E}(M)$, and its reverse operation: the \emph{desingularization}. Given an element in $\partial_o\mathbf{E}(M)$, the desingularization consists in finding a $d_{GH}$-approximating sequence of metrics in $\mathbf{E}(M)$.

\emph{Any} smooth Einstein $4$-manifold close to a compact Einstein orbifold in a mere Gromov-Hausdorff sense has been recently been produced by a gluing-perturbation procedure \cite{ozu1,ozu2}. Here, this lets us understand the obstructions to the $d_{GH}$-desingularization of an Einstein orbifold through obstructions similar to \eqref{obst killing intro} and \eqref{obst conformal intro}. 
\begin{thm}[Theorem \ref{obstruction desing taub's}]\label{obst desing integr}
    Assume that there exists a sequence of Einstein metrics $(M,\mathbf{g}_n)_n$ converging in the Gromov-Hausdorff sense to an Einstein orbifold $(M_o,\mathbf{g}_o)$ with a singularity $\mathbb{R}^4\slash\Gamma$ at $p$ and satisfying $\Ric(\mathbf{g}_o) = \Lambda\mathbf{g}_o$. Assume that there exist numbers $t_n>0$ such that at $p$, $(M,\mathbf{g}_n/t_n)_n$ converges to a Ricci-flat ALE \emph{manifold} $(N,\mathbf{g}_b)$ asymptotic to $\mathbb{R}^4\slash\Gamma$ with \emph{integrable} infinitesimal deformations. Consider the following asymptotics in well-chosen gauges (in a so-called \emph{volume gauge} for $\mathbf{g}_b$):
    $$ \mathbf{g}_o = \mathbf{e} + H_2 + \mathcal{O}(r^3) \;\text{ and }\;\mathbf{g}_b = \mathbf{e} + H^4 + \mathcal{O}(r^{-5}),$$
    for $|H_2|_\mathbf{e} \sim r^2$ and $|H^4|_\mathbf{e} \sim r^{-4}$.
    
    Then, the following obstructions analogous to \eqref{obst killing intro} and \eqref{obst conformal intro} hold: for any $Y$ Killing vector field of $\mathbb{R}^4\slash\Gamma$:
    \begin{equation}
        \int_{\mathbb{S}^{3}/\Gamma} \big(\mathring{\Ric}^{(2)}_{\mathbf{e}}(H^4,H_2)\big)(Y,\partial_r)  dv_{\mathbb{S}^{3}/\Gamma}=0,\label{obst desing intro killing}
    \end{equation}
    and
    \begin{equation}
         \int_{\mathbb{S}^{3}/\Gamma} \big(\mathring{\Ric}^{(2)}_{\mathbf{e}}(H^4,H_2)\big)(r\partial_r,\partial_r)  dv_{\mathbb{S}^{3}/\Gamma}=0.\label{obst desing intro radial}
    \end{equation}
\end{thm}
Comparing \eqref{obst killing intro} to \eqref{obst desing intro killing} and \eqref{obst conformal intro} to \eqref{obst desing intro radial}, we therefore interpret the obstructions to the desingularization of an Einstein orbifold as a defect of integrability of the infinitesimal Einstein deformation $h=H_2+H^4$. The proof of these different obstructions actually relies on the control of the same integration by parts as Theorem \ref{thm intro} once one notices that in so-called \emph{volume gauge}, the vector field $r\partial_r$ is $\mathbf{g}_b$-harmonic at an order higher than expected.

\begin{rem}
    Denote $ \mathcal{V}(\mathbf{g}_b) <0$ the \emph{reduced volume} of $(N,\mathbf{g}_b)$ introduced in \cite{bh}, as well as $ W^\pm_{\mathbf{g}_o}$ the Weyl curvatures of $\mathbf{g}_o$ at $p$ and $W^\pm_{\mathbf{g}_b}$ the asymptotic Weyl curvatures (the part decaying like $r^{-6}$) of $\mathbf{g}_b$, then \eqref{obst desing intro radial} may rewrite as:
    \begin{equation}
        \Lambda\mathcal{V}(\mathbf{g}_b)  + Q\big(W^+_{\mathbf{g}_b}, W^-_{\mathbf{g}_o}\big) + Q\big(W^-_{\mathbf{g}_b}, W^+_{\mathbf{g}_o}\big)=0\label{obst weyl volume}
    \end{equation}
    for some explicit quadratic form $Q$. A similar rewriting of \eqref{obst desing intro killing} yields $Q'\big(W^+_{\mathbf{g}_b}, W^-_{\mathbf{g}_o}\big) + Q'\big(W^-_{\mathbf{g}_b}, W^+_{\mathbf{g}_o}\big)=0$ for some other explicit quadratic form $Q'$.
\end{rem}

    \begin{rem}\label{obst pas sat sph}
        For $\mathbf{g}_o$ either spherical of hyperbolic, one has $\Lambda \neq 0$ and $W^\pm_{\mathbf{g}_o} =0$. The obstruction \eqref{obst weyl volume} is therefore \emph{never} satisfied. We also recover that the obstruction vanishes for the gluing of a hyperkähler ALE to a hyperkähler ALE since $\Lambda = 0$, $W^+_{{\mathbf{g}_o}} = 0$ and $W^+_{\mathbf{g}_b}=0$.
    \end{rem}
    We recover \emph{all} of the obstructions to the desingularization by Eguchi-Hanson metrics identified in \cite{biq1}, see Corollary \ref{obst EH} and an extension to higher dimension yields the obstruction of \cite{mv} as well, see Corollary \ref{mv obst}. In these articles, the obstructions were purely \emph{analytical} as projections on the cokernel of the linearized operator. Theorem \ref{obst desing integr} gives a new maybe more \emph{geometric} interpretation of them.
    
    \begin{rem}
        There are higher order obstructions to the existence of Einstein deformations with asymptotic developments $\mathbf{g}_t =\mathbf{g} + th_{1} + t^2h_2 +...$ for small $t$ which are very similar to the obstructions \eqref{obst desing intro radial} and \eqref{obst desing intro killing}.
    \end{rem}
    
    \begin{quest}
    Can these higher order obstructions recover the higher order obstructions of \cite{ozu3}? Can they help compute even higher order obstructions?
    \end{quest}
\subsection*{Desingularization of spherical and hyperbolic orbifolds}
    We next get to the main application of this article. We answer negatively the classical question of whether or not \emph{all} Einstein $4$-orbifolds can be $d_{GH}$-desingularized by smooth Einstein $4$-manifolds.
    
\begin{thm}\label{obst sph hyp}
    A spherical or hyperbolic $4$-orbifold with at least one singularity $\mathbb{R}^4\slash\mathbb{Z}_2$ cannot be limit of smooth Einstein metrics in the Gromov-Hausdorff sense.
\end{thm}
    
\begin{exmp}
    Consider $\mathbb{S}^4\subset \mathbb{R}^5$ and the quotient by $\mathbb{Z}_2$ given by $(x_1,x_2,x_3,x_4,x_5) \sim (x_1,-x_2,-x_3,-x_4,-x_5)$. We will denote this space $\mathbb{S}^4\slash\mathbb{Z}_2$ which is an Einstein orbifold with two $\mathbb{R}^4\slash\mathbb{Z}_2$ singularities. It is often called the \emph{American football metric}. It is a \emph{synthetic} Einstein space in the sense of \cite{nab} for instance. By the above Theorem \ref{obst sph hyp}, it cannot be a Gromov-Hausdorff limit of smooth Einstein metrics.
\end{exmp}
Theorem \ref{obst sph hyp} was conjectured in the author's PhD thesis \cite{ozuthese} where it was proven under a technical assumption of integrability for the Ricci-flat ALE spaces. The main remaining difficulty here is thus to deal with the potential non integrability of these deformations. We more precisely prove that Theorem \ref{obst desing integr} holds without the integrability assumption on the Ricci-flat ALE spaces. Remark \ref{obst pas sat sph} then lets us conclude.

\begin{rem}
    It is often conjectured that the only Ricci-flat ALE metrics are Kähler, hence integrable. However, the motivation of this conjecture formulated in \cite{bkn} for instance, seems to be the analogous conjecture for the selfduality of Yang-Mills connections on $SU(2)$ bundles over $\mathbb{S}^4$. This analogous conjecture was disproved the same year in \cite{ssu}. 
\end{rem}

Let us present the main step of the proof of Theorem \ref{obst sph hyp} which is of independent interest. Let $(N,\mathbf{g}_b)$ be a Ricci-flat ALE orbifold asymptotic to $\mathbb{R}^4\slash\Gamma$ for $\Gamma\subset SO(4)$. Its space of $L^2$-infinitesimal deformations which are traceless and in divergence-free gauge is denoted $\mathbf{O}(\mathbf{g}_b)$. 
There are particular elements in $\mathbf{O}(\mathbf{g}_b)$ coming from the symmetries of the asymptotic cone $\mathbb{R}^4\slash\Gamma$. According to \cite{ozuthese}:
\begin{itemize}
    \item there exists $X$ a harmonic vector field on $(N,\mathbf{g}_b)$ asymptotic to the \emph{conformal Killing vector field} $r\partial_r$, and 
    $(\mathcal{L}_X\mathbf{g}_b)^\circ\in \mathbf{O}(\mathbf{g}_b),$
    where for a symmetric $2$-tensor $h$, $ (h)^\circ$ denotes the traceless part of $h$,
    \item for any \emph{Killing vector field} $Y$ there exists $Y'$ a harmonic vector field on $(N,\mathbf{g}_b)$ asymptotic to $Y$, and 
    $\mathcal{L}_{Y'}\mathbf{g}_b\in \mathbf{O}(\mathbf{g}_b).$
\end{itemize}

The main step in the proof of Theorem \ref{obst sph hyp} is to show that these deformations which are elements of the kernel of the linearization of the Ricci curvature should not be thought of as elements of the cokernel when dealing with Ricci-flat ALE deformations, at least at leading order. More precisely, for $v\in\mathbf{O}(\mathbf{g}_b)$, let us consider $ g_v $ the unique solution to:
\begin{equation}
    \begin{aligned}
        \mathbf{\Phi}_{\mathbf{g}_b}(g_v) = \E(g_v) + \delta^*_{\mathbf{g}_b}\delta_{\mathbf{g}_b}g_v\in \mathbf{O}(\mathbf{g}_b).
    \end{aligned}
\end{equation}
satisfying $g_v - (\mathbf{g}_b+v)\perp_{L^2(\mathbf{g}_b)} \mathbf{O}(\mathbf{g}_b) $. We call \emph{Einstein modulo obstructions metrics} such deformations which have been constructed in \cite{ozu2} (see also \cite{koi} in the smooth compact case). We study the leading order of the obstruction along curves $s\mapsto g_{sv}$ for $s\in (-1,1)$ at $s=0$.

\begin{prop}\label{prop orthogonal lie der}
    Let $(N,\mathbf{g}_b)$ be a Ricci-flat ALE metric which has \emph{nonintegrable} Ricci-flat ALE deformations. Then, for any $v\in \mathbf{O}(\mathbf{g}_b)$, there exists $l\geqslant 2$ such that  $\partial^k_{s^k|s=0}\mathbf{\Phi}_{\mathbf{g}_b}(g_{sv}) = 0$ for all $k\leqslant l-1$ and $\partial^l_{s^l|s=0}\mathbf{\Phi}_{\mathbf{g}_b}(g_{sv}) \neq  0$. The leading order obstruction $\partial^l_{s^l|s=0}\mathbf{\Phi}_{\mathbf{g}_b}(g_{sv})$ is $L^2(\mathbf{g}_b)$-orthogonal to the vector subspace of $\mathbf{O}(\mathbf{g}_b)$ spanned by the above elements $(\mathcal{L}_X\mathbf{g}_b)^\circ$ and $\mathcal{L}_{Y'}\mathbf{g}_b$.
\end{prop}
The proof relies on careful integrations by parts similar to that of the proof of Theorem \ref{thm intro}. 

\subsection*{Acknowledgements}
The author would like to thank Olivier Biquard and Hans-Joachim Hein for inspiring discussions and for sharing some of their early results which lead to \cite{bh}.

\section{Symmetries of Einstein metrics and integrability}

In this section, we recall well-known properties of the two first derivatives of the Einstein operator and apply them to define the so-called \emph{Taub's preserved quantity} as introduced in \cite{tau}. It is a central quantity in the study of the integrability of Einstein deformations in the Lorentzian context.

\begin{note}
    All along the article, we will denote by $F_g^{(m)}$ the $m$-linear terms of the development of a functional $h\mapsto F(g+h)$ at $0$, we more precisely have (at least formally) for any small enough $2$-tensor $h$:
    $$ F(g+h) = \sum_{m\in\mathbb{N}}F_g^{(m)}\underbrace{(h,...,h)}_{m \text{ times }}.$$
\end{note}
\subsection{Gauge and reparametrization properties}
Let us start by recalling various consequences of the Bianchi identity: for any Riemannian metric $g$, one has
\begin{equation}
    B_g(\Ric(g)) = 0.\label{bianchi identity}
\end{equation}
where for any $2$-tensor $h$, we defined the \emph{Bianchi operator} $ B_g = \delta_g\big(h - \frac{1}{2}\textup{tr}(h)g \big) = \delta_g + \frac{1}{2}d\textup{tr}(h)$, where $\delta_g$ is the divergence with the convention that in coordinates, for a $1$-form $\omega$, $\delta_g\omega := -g^{ij}\nabla_j\omega_{i}$ and for a symmetric $2$-tensor $h$, $(\delta_gh)_k := -g^{ij}\nabla_jh_{ik}$. Denoting the Einstein tensor $\E(g):=\Ric(g) - \frac{\R_g}{2}g$, the equation \eqref{bianchi identity} rewrites
\begin{equation}
    \delta_g(\E(g)) = 0.\label{Einstein divergence identity}
\end{equation}
\begin{prop}[{\cite{fmm}}]\label{gauge properties Bianchi}
    Let us assume that $\mathbf{g}$ is a Ricci-flat metric on some open domain $\mathcal{U}$ of a Riemannian manifold $M$, and let $h$ be a symmetric $2$-tensor on $\mathcal{U}$. We have the following gauge properties:
    \begin{itemize}
        \item without assumption, one has 
        \begin{equation}
            B_\mathbf{g}\big(\Ric_\mathbf{g}^{(1)}(h)\big)= 0, \text{ and } \delta_\mathbf{g}\big(\E_\mathbf{g}^{(1)}(h)\big)= 0,\label{bianchi div 1er ordre}
        \end{equation}
        \item if $\Ric_\mathbf{g}^{(1)}(h)=\Lambda g$ for $\Lambda\in \mathbb{R}$ or equivalently $\E_\mathbf{g}^{(1)}(h)+ \lambda g=0$ for $\lambda\in \mathbb{R}$, then one has
        \begin{equation}
          B_\mathbf{g}\big(\Ric_\mathbf{g}^{(2)}(h,h)\big)= 0, \text{ and } \delta_\mathbf{g}\big(\E_\mathbf{g}^{(2)}(h,h)\big)= 0.\label{bianchi div 2nd ordre}
        \end{equation}
    \end{itemize}
\end{prop}
\begin{proof}[ ]
    
\end{proof}

Let us continue with some identities when $g$ is perturbed in the direction of a Lie derivative. These come from differentiation of the following identity: for any diffeomorphism $\phi:M\to M$ and any metric $g$ on M
\begin{equation}
    \Ric(\phi^*g) = \phi^*(\Ric(g)) \text{ and } \E(\phi^*g) = \phi^*(\E(g)).
\end{equation}

\begin{prop}[{\cite{fmm}}]\label{covariance reparam}
Let us assume that $g$ is a Riemannian metric on some bounded open domain $\mathcal{U}$ of a Riemannian manifold $M$, let $h$ be a symmetric $2$-tensor on $\mathcal{U}$ and $X$ be a vector field on $\mathcal{U}$. We have the following reparametrization properties for the derivatives of $\Ric$: without assumption, one has 
        \begin{equation}
            \Ric_g^{(1)}(\mathcal{L}_Xg)= \mathcal{L}_X(\Ric(g)) \text{ and } \E_g^{(1)}(\mathcal{L}_Xg)= \mathcal{L}_X(\E(g)),
        \end{equation}
        \begin{equation}
            \Ric_g^{(2)}(h,\mathcal{L}_Xg) + \Ric_g^{(1)}(\mathcal{L}_Xh)= \mathcal{L}_X(\Ric_g^{(1)}(h)),
        \end{equation}
        and
        \begin{equation} \E_g^{(2)}(h,\mathcal{L}_Xg) + \E_g^{(1)}(\mathcal{L}_Xh)= \mathcal{L}_X(\E_g^{(1)}(h))\label{gauge and E2}
        \end{equation}
\end{prop}
\begin{proof}[ ]

\end{proof}

\subsection{Taub's preserved quantities and obstructions}

Let us start by stating the following classical integrations by parts which are at the core of our proofs.
\begin{lem}\label{divergence killing}
    Let $g$ be a metric on an open set $\mathcal{U}$, $T$ be a \emph{divergence-free} symmetric $2$-tensor, and let $X$ be a vector field on $\mathcal{U}$. Then for any smooth compact subset $\Omega\subset \mathcal{U}$ with boundary, we have:
    \begin{equation}
        \int_{\partial\Omega}T(X,n) dv_{g_{|\partial\Omega}}= \int_{\Omega}\left\langle T,\delta_g^*X \right\rangle_g dv_g =  \frac{1}{2}\int_{\Omega}\left\langle T,\mathcal{L}_Xg \right\rangle_g dv_g,\label{ipp div free 2tensor}
    \end{equation}
    where $n$ is the outward unit normal to $\partial\Omega$, and where $\delta_g^*X := \frac{1}{2}\mathcal{L}_{X}g$ is the formal adjoint of the divergence $\delta_g$.
    
    We also have the following identity close to Schoen's Pohozaev equality \cite{sch}. Denote $ \mathring{h} $ or $(h)^\circ$ the traceless part of a symmetric $2$-tensor $h$. If $T$ is divergence-free, then we have:
    \begin{equation}
    \begin{aligned}
        \int_{\partial\Omega}\mathring{T}(X,n) dv_{g_{|\partial\Omega}} &=  \frac{1}{2}\int_{\Omega}\Big(\left\langle \mathring{T},\mathcal{L}_Xg  \right\rangle_g - \frac{\mathcal{L}_X (\tr_g T)}{d}\Big) dv_g \\
        &=  \frac{1}{2}\int_{\Omega}\Big(\left\langle T,(\mathcal{L}_Xg)^\circ \right\rangle_g- \frac{\mathcal{L}_X (\tr_g T)}{d}\Big) dv_g,
    \end{aligned}\label{schoen poho}
    \end{equation}
    
    \end{lem}
    
    \begin{rem}
        We will often abusively apply our operators to vector fields or $1$-forms indifferently, the identification will always be done thanks to the metric involved in the operator. More precisely, a vector field $X$ is identified with the $1$-form $g(X,.)$.
    \end{rem}
\begin{proof}
    The key to this formula is the identity:
    \begin{equation}
        \delta_g\big(T(X)\big)= \delta_g(T)(X) - \langle T,\delta_g^*X \rangle_g = \delta_g(T)(X) - \frac{1}{2} \langle T,\mathcal{L}_Xg \rangle_g\label{divergence }
    \end{equation}
    which may be proven in coordinates using the symmetry of $T$. From the identity \eqref{divergence }, using the fact that $T$ is divergence-free and the divergence theorem, we find the result by the divergence theorem.
    
    The second equality \eqref{schoen poho} then follows by noting that $$ \int_\Omega \frac{\tr_g T}{n}\langle g, \mathcal{L}_Xg \rangle_gdv_g = -\frac{1}{n} \int_\Omega(\tr_g T) \delta_g(X) dv_g = -\int_\Omega\mathcal{L}_X(\tr_g T) dv_g + \int_{\partial\Omega} (\tr_g T) g(X,n) .$$
\end{proof}

Let $\mathbf{g}$ be an Einstein metric on an open subset $\mathcal{U}$. For a vector field $X$, a closed orientable hypersurface $\Sigma \subset \mathcal{U}$ and symmetric $2$-tensors $h$ and $k$ on $\mathcal{U}$ we define the following quantity:
\begin{equation}
    \mathcal{B}_X^\Sigma(h):=\int_\Sigma \big(\E^{(1)}_\mathbf{g}(h)\big)(X,n_\Sigma)dv_\Sigma,\label{eq def int Ein 1}
\end{equation}
where $n_\Sigma$ is the normal to $\Sigma$. We also define the so-called \emph{Taub's preserved quantity} introduced in \cite{tau}:
\begin{equation}
    \mathcal{T}_X^\Sigma(h,k):=\int_\Sigma \big(\E^{(2)}_\mathbf{g}(h,k)\big)(X,n_\Sigma)dv_\Sigma.\label{eq def taub}
\end{equation}
Together with the gauge properties of Proposition \ref{gauge properties Bianchi}, we use Lemma \ref{divergence killing} to prove the following properties for $\mathcal{T}$ and $\mathcal{B}$.
\begin{prop}\label{prop B T}
    Let $\mathbf{g}$ be an Einstein metric on an open subset $\mathcal{U}$, $ X $ be a Killing vector field, $h$ a $2$-tensor on $\mathcal{U}$, and $\Sigma$ and $\Sigma'$ two closed  hypersurfaces in $\mathcal{U}$ bounding an open subset $\Omega\subset\mathcal{U}$. Then, we have the following properties:
    \begin{enumerate}
        \item without additional assumption, 
        \begin{equation}
            \mathcal{B}_X^\Sigma(h) = 0\label{annulation B}
        \end{equation}
        \item if  $\E^{(1)}_\mathbf{g}(h)=0 $, then one has: $
            \mathcal{T}_X^{\Sigma'}(h,h) = \mathcal{T}_{X}^\Sigma(h,h)$
        \item if  $\E^{(1)}_\mathbf{g}(h)=0 $, and for any vector field $Y$ on $\mathcal{U}$, we have $
            \mathcal{T}_X^{\Sigma}(h + \mathcal{L}_Y\mathbf{g},h+ \mathcal{L}_Y\mathbf{g}) = \mathcal{T}_{X}^\Sigma(h,h).$
    \end{enumerate}
\end{prop}
\begin{proof}
    Consider $\chi$ a cut-off function vanishing in the neighborhood of $\Sigma'$ and equal to $1$ on a neighborhood of $\Sigma$. We can therefore apply \eqref{ipp div free 2tensor} on $\Omega$ bounded by $\Sigma$ and $\Sigma'$ to the Killing vector field $X$ and $T = \E^{(1)}(\chi h)$ which is divergence-free by \eqref{bianchi div 1er ordre} to find:
    $$ \mathcal{B}_X^\Sigma(h) =\int_{\partial\Omega} \big(\E^{(1)}(\chi h)\big)(X,n)dv_{\partial\Omega} = 0.$$
    See \cite{fmm} for the other equalities.
\end{proof}

\subsection{Integrability of Einstein deformations}

In the Lorentzian context, it is a remarkable result that the quantities $\mathcal{T}_\Sigma^X(h,h)$ for the different Killing vector fields $X$ of $\mathbf{g}$ and $h$ satisfying $\E^{(1)}_\mathbf{g}(h)=0$ completely characterize the integrability of the infinitesimal Einstein deformation $h$. 

\begin{defn}[Integrable $2$-tensor]
    Let $\mathbf{g}$ be an Einstein metric on $\mathcal{U}$.
    A $2$-tensor $h$ is \emph{integrable} if on any compact $K\subset \mathcal{U}$, there exists a smooth curve of Einstein metrics $t\in[0,1]\mapsto\mathbf{g}_t$ on $K$ satisfying $ \partial_t{\mathbf{g}_t}_{|t=0}=h $.
\end{defn}

The link between the integral quantity $\mathcal{T}_\Sigma^X(h,h)$ and the integrability of $h$ is given by the following proposition.

\begin{prop}[{\cite[Proposition 1.7]{fmm}}]\label{integrability taub}
    Assume that $\mathbf{g}$ is a Ricci-flat metric on an open subset $\mathcal{U}$, that $h$ is an integrable $2$-tensor which is the first jet of a curve of Ricci-flat metrics and $X$ is a Killing vector field. Then, for any compact hypersurface $\Sigma\subset \mathcal{U},$ one has 
    \begin{equation}
        \mathcal{T}_X^\Sigma(h,h)=0.\label{obst taub}
    \end{equation}
\end{prop}
\begin{proof}
    Let $K$ be a compact subset of $\mathcal{U}$ and $t\in[0,1]\mapsto\mathbf{g}_t$ a smooth curve of Einstein metrics on $K$ with $\mathbf{g}_0=\mathbf{g}$ satisfying $\partial_t{\mathbf{g}_t}_{|t=0}=h$ with $\E(\mathbf{g}_t) = 0 $ for all $t\in[0,1]$. We have 
    $ 0=\partial_t (\E(\mathbf{g}_t))_{|t=0} = \E_\mathbf{g}^{(1)}(h), $
    and if we denote $ k:= \partial^2_{t^2}{\mathbf{g}_t}_{|t=0}$, then we have
    $$ 0=\partial^2_{t^2} (\E(\mathbf{g}_t))_{|t=0} = \E_\mathbf{g}^{(1)}(k) + \E_\mathbf{g}^{(2)}(h,h). $$
    In particular, we have:
    $\mathcal{T}_X^\Sigma (h,h) = -\mathcal{B}_X^\Sigma (k)=0 $
    by \eqref{annulation B}.
\end{proof}
We add an extension to Einstein but not necessarily Ricci-flat deformation of a Ricci-flat metric. This time we need to assume that the Killing vector field is tangent to our hypersurface. This will always be satisfied in our applications.

\begin{prop}\label{integrability taub einstein}
    Assume that $\mathbf{g}$ is a Ricci-flat metric on an open subset $\mathcal{U}$, that $h$ is a $2$-tensor which is the first jet of a smooth curve of Einstein metrics starting at $\mathbf{g}$ and $X$ is a Killing vector field for $\mathbf{g}$. Then, for any compact hypersurface $\Sigma\subset \mathcal{U}$ for which $X$ is tangent to $\Sigma$ once restricted to $\Sigma$, one has 
    \begin{equation}
        \int_{\Sigma}\big( \mathring{\Ric}^{(2)}_\mathbf{g}(h,h)\big)(X,n_\Sigma)dv_\Sigma=0.\label{obst taub einstein}
    \end{equation}
    \end{prop}
\begin{proof}
    Let $K\subset \mathcal{U}$ be a compact and $t\in[0,1]\mapsto\mathbf{g}_t$ a smooth curve of Einstein metrics on $K$ with $\mathbf{g}_0=\mathbf{g}$ satisfying $\partial_t{\mathbf{g}_t}_{|t=0}=h$ with $\E(\mathbf{g}_t) +\lambda(t)\mathbf{g}_t=0 $ for a smooth function $t\in[0,1]\mapsto \lambda(t)\in \mathbb{R}$ for all $t\in[0,1]$. We have 
    $$ 0=\partial_t (\E(\mathbf{g}_t) +\lambda(t)\mathbf{g}_t)_{|t=0} = \E_\mathbf{g}^{(1)}(h)+\lambda'(0)\mathbf{g}, $$
    and if we denote $ k:= \partial^2_{t^2}{\mathbf{g}_t}_{|t=0}$, then we find
    $$ 0=\partial^2_{t^2} (\E(\mathbf{g}_t) +\lambda(t)\mathbf{g}_t)_{|t=0} = \E_\mathbf{g}^{(1)}(k) + \E_\mathbf{g}^{(2)}(h,h) +\lambda''(0)\mathbf{g}+2\lambda'(0)h. $$
    In particular, since for any metric $g$, one has $\mathring{\Ric}(g) = \E(g) -\frac{\tr_g\E(g)}{d}g$ in dimension $d$, we find:
    \begin{align}
        \mathring{\Ric}^{(2)}_\mathbf{g}(h,h) &= \E^{(2)}_\mathbf{g}(h,h) -\frac{1}{d}\big((\tr \E)_{\mathbf{g}}^{(2)}(h,h)\big)\mathbf{g} -\frac{2}{d}\tr_\mathbf{g}\E^{(1)}_\mathbf{g}(h) h\nonumber \\
      &=\E^{(2)}_\mathbf{g}(h,h) -\frac{1}{d}\big((\tr \E)_{\mathbf{g}}^{(2)}(h,h)\big)\mathbf{g}+2\lambda'(0) h.\label{traceless Ric2}
    \end{align}
    We consequently have: $
        \int_{\Sigma}\big( \mathring{\Ric}^{(2)}_\mathbf{g}(h,h)\big)(X,n_\Sigma)dv_\Sigma = -\int_{\Sigma} \big(\E^{(1)}_\mathbf{g}(k)\big)(X,n_\Sigma)dv_\Sigma = 0$
    because $\mathbf{g}(X,n_\Sigma)=0$ by assumption and thanks to \eqref{annulation B}.
\end{proof}
\begin{quest}
    Can we recover the obstructions to integrability of \cite{koi2} (or new ones) thanks to the above criterion?
\end{quest}

\section{Preserved quantities on Ricci-flat cones and integrability}

We introduce another quantity similar to Taub's adapted to perturbations of Ricci-flat cones $(C(\Sigma),dr^2+r^2g_\Sigma)$. We will mostly work with the Euclidean case $\Sigma=\mathbb{S}^{d-1}$.

\subsection{A preserved quantity on $\mathbb{R}^d$ and other Ricci-flat cones}

Let us now introduce a functional similar to Taub's preserved quantity dealing this time with \emph{conformal} Killing vector fields with constant conformal factor. The main example for us will be the vector field $r\partial_r$ on the Euclidean space $(\mathbb{R}^d,\mathbf{e})$, where $r := d_\mathbf{e}(0,.)$.
\begin{prop}\label{premiere obst conforme}
    Let $h$ be a symmetric $2$-tensor on an open subset $\mathcal{U}\subset\mathbb{R}^d$ containing $\mathbb{S}^{d-1}$. Then, we have the following identity:
    \begin{equation}
        \int_{\mathbb{S}^{d-1}} \big(\E^{(1)}_\mathbf{e}(h)\big)(r\partial_r,\partial_r)+ \frac{d-2}{2} \big(-\delta_\mathbf{e}h-d\tr_\mathbf{e}h\big)(\partial_r) dv_{\mathbb{S}^{d-1}}=0. \label{B cas euclidien}
    \end{equation}
    Moreover, assume that we have $\R^{(1)}_\mathbf{e}(h)= 0$ and $h$ bounded on $B_\mathbf{e}(1+\epsilon)$. Then we have:
    \begin{equation}
        \int_{\mathbb{S}^{d-1}} \big(\E^{(1)}_\mathbf{e}(h)\big)(r\partial_r,\partial_r)=0. \label{B cas euclidien R1 nul}
    \end{equation}
\end{prop}
\begin{rem}
    One recognizes (up to a constant) the integrand of the ADM mass in the term $\int_{\mathbb{S}^{d-1}} \big(-\delta_\mathbf{e}h-d\tr_\mathbf{e}h\big)(\partial_r) dv_{\mathbb{S}^{d-1}}$. This is not surprising from the proof and can be seen as a first order version of the proof of the equality between the mass and the so-called \emph{Ricci version of the mass} given in \cite{her}.
\end{rem}
\begin{proof}
    As in the proof of the first point in Proposition \ref{prop B T}, without changing the value of the integral on $\mathbb{S}^{d-1}$, we replace $h$ by another $2$-tensor equal to $h$ in a neighborhood of $\mathbb{S}^{d-1}$ and vanishing in a neighborhood of the sphere $(1-\epsilon)\mathbb{S}^{d-1}\subset \mathcal{U}$ for $\epsilon>0$ small enough. We denote $\Omega$ the open subset bounded by $\mathbb{S}^{d-1}$ and $(1-\epsilon)\mathbb{S}^{d-1}$.

    From the identity \eqref{bianchi div 1er ordre} and Lemma \ref{divergence killing}, we can use the divergence theorem and the equality $\delta^*_\mathbf{e}(r\partial_r) = \mathbf{e}$ to find the identity:
    \begin{equation}
        \int_{\mathbb{S}^{d-1}} \big(\E^{(1)}_\mathbf{e}(h)\big)(r\partial_r,\partial_r)dv_{\mathbb{S}^{d-1}} = \int_\Omega \tr_\mathbf{e} \E^{(1)}_\mathbf{e}(h) dv_\mathbf{e}.\label{identite E(1)}
    \end{equation}
    Now, one recognizes that $\tr_\mathbf{e} \E^{(1)}_\mathbf{e}(h) = \frac{2-d}{2} \R^{(1)}_\mathbf{e}(h)$ by definition of the scalar curvature $\R_g=\tr_g\Ric_g$ and since $\E_\mathbf{e}=0$. From the first variation of the scalar curvature at a Euclidean metric:
    $$ \R^{(1)}_\mathbf{e}(h) = \delta_\mathbf{e}\big(\delta_\mathbf{e}(h) +d\tr_\mathbf{e}(h)\big), $$
    using the divergence theorem, we find the stated formula.
    
    Finally, if $h$ is defined on $B_\mathbf{e}(1+\epsilon)$ for some $\epsilon>0$ and satisfies $\R_\mathbf{e}^{(1)}(h)=0$ then one has $\int_{\mathbb{S}^{d-1}} \big(\E^{(1)}_\mathbf{e}(h)\big)(r\partial_r,\partial_r)dv_{\mathbb{S}^{d-1}}$ by \eqref{identite E(1)}.
\end{proof}

One finds an obstruction similar to that of Proposition \ref{integrability taub} for the second order variation of the metric direction.

\begin{prop}\label{seconde obst conforme}
    Let $h$ be a symmetric $2$-tensor on an open subset $\mathcal{U}\subset\mathbb{R}^d$ containing $\mathbb{S}^{d-1}$ and $B_\mathbf{e}(1)$ satisfying $\E_\mathbf{e}^{(1)}(h) = 0$. Then, we have the following identity:
    \begin{equation}
    \begin{aligned}
        \int_{\mathbb{S}^{d-1}}& \big(\E^{(2)}_\mathbf{e}(h,h)\big)(r\partial_r,\partial_r)\\
        &+ \frac{2-d}{2}\Big[\big((\delta+d\tr)^{(1)}_\mathbf{e}(h)\big)(v) +\frac{\tr_\mathbf{e}h}{2}(\delta_\mathbf{e}v+d\tr_\mathbf{e}v) \Big](\partial_r) dv_{\partial\Omega}=0.
    \end{aligned}
         \label{annulation T cas euclidien}
    \end{equation}
    The same result is true if $h$ is defined on an open subset $\mathcal{U}\subset\mathbb{R}^d$ containing $\mathbb{S}^{d-1}$ and $\mathbb{R}^d\backslash B_\mathbf{e}(1)$ and if it additionally satisfies the decay assumption: 
    $r^k|\nabla^k_\mathbf{e}h|_\mathbf{e}\leqslant r^{-\frac{d-2}{2}}$
    for $k\in \{0,1,2\}$.
\end{prop}
\begin{rem}\label{expression delta(1) tr(1)}
    For $h$ and $k$ symmetric $2$-tensors, we have  $$\Big[\big(\delta^{(1)}_\mathbf{e}(h)\big)(k)\Big](\partial_r) = -\delta_\mathbf{e}(h\times k)(\partial_r)-\frac{1}{2}k(d\tr_\mathbf{e}h,\partial_r)+\frac{1}{2}\langle k, \nabla_{\partial_r}h \rangle_\mathbf{e}$$
    where $h\times k$ is the $2$-tensor obtained from the composition of the endomorphisms associated to h and $k$ and
    $$\Big[\big(d\tr^{(1)}_\mathbf{e}(h)\big)(k)\Big](\partial_r) =- d\big(\langle h,k \rangle_\mathbf{e}\big)(\partial_r),$$
    see \cite{cox} for instance. We will try not to use these formulas whenever possible and focus on situations for which this additional boundary term vanishes.
\end{rem}

\begin{rem}
    The two situations of Proposition \ref{seconde obst conforme} model Einstein manifolds or orbifolds at a given point or Einstein metrics asymptotic to a quotient of $\mathbb{R}^d$ at infinity. 
    The decay condition on $h$ is not strong as any $h$ decaying at any rate at infinity and in divergence-free gauge (or Bianchi gauge) which satisfies $\E^{(1)}_\mathbf{e}(h)=0$ automatically satisfies $r^k|\nabla^k_\mathbf{e}h|_\mathbf{e}\leqslant r^{-d+1}$ by \cite{ct}.
\end{rem}

\begin{proof}
    Let us assume that $\E_\mathbf{e}^{(1)}(h)= 0$ on an open set containing $B_\mathbf{e}(1)$. Using the formula \eqref{ipp div free 2tensor} with $T = \E_\mathbf{e}^{(2)}(h,h)$ and $\Omega = B_\mathbf{e}(1)$, we find the equality:
    \begin{equation}
        \begin{aligned}
     \int_{\mathbb{S}^{d-1}}& \big(\E^{(2)}_\mathbf{e}(h,h)\big)(r\partial_r,\partial_r)dv_{\mathbb{S}^{d-1}}= \int_\Omega \tr_\mathbf{e}\E^{(2)}_\mathbf{e}(h,h) dv_{\mathbf{e}}
        \end{aligned}
    \end{equation}
    where we used $\delta_\mathbf{e}^*(r\partial_r)=\mathbf{e}$ and the fact that there is only \emph{one} nonvanishing boundary term at $\mathbb{S}^{d-1}$.

    There remains to understand the term $ \int_\Omega \tr_\mathbf{e}\E^{(2)}_\mathbf{e}(h,h) dv_{\mathbf{e}} $ as a boundary term. Since $\E_\mathbf{e}^{(1)}(h)=0$, then, $\tr_\mathbf{e}\E^{(2)}_\mathbf{e}(h,h) = \frac{2-d}{2} \R^{(2)}_\mathbf{e}(h,h)$ where we recall that the formula for $\R^{(1)}$ is given by
        $\R_g^{(1)}(v) = \delta_g\left(\delta_gv+d\tr_g v\right) -\left\langle v,\Ric(g)\right\rangle$
        for a general metric $g$ and a deformation $v$. This implies that one generally has:
        \begin{equation}
           \int_{\Omega}\R_g^{(1)}(v)dv_g = -\int_{\partial\Omega}\big(\delta_gv+d\tr_g v\big)(\partial_r) dv_{\partial\Omega} -\int_\Omega \left\langle v,\Ric(g)\right\rangle dv_g.\label{ipp courbure scalaire(1) general}
        \end{equation}
        
        By differentiating \eqref{ipp courbure scalaire(1) general} with $\Omega=B_\mathbf{e}(1)$ at $g=\mathbf{e}$ in the direction $h$ and for $v$ satisfying $\E_\mathbf{e}^{(1)}(v)=\Ric^{(1)}_\mathbf{e}(v) = 0$ we obtain:
        \begin{equation}
           \int_{B_\mathbf{e}(1)}\R_\mathbf{e}^{(2)}(h,v)dv_\mathbf{e} = -\int_{\mathbb{S}^{d-1}}\Big[\big((\delta+d\tr)^{(1)}_\mathbf{e}(h)\big)(v) +\frac{\tr_\mathbf{e}h}{2}(\delta_\mathbf{e}v+d\tr_\mathbf{e}v) \Big](\partial_r) dv_{\mathbb{S}^{d-1}}.\label{ipp courbure scalaire(2) general}
        \end{equation}
        Using \eqref{ipp courbure scalaire(2) general} with $v=h$ yields the result.
         
        The proof is exactly the same when $\Omega = \mathbb{R}^d\backslash B_\mathbf{e}(1)$ if we assume that $r^k|\nabla^k_\mathbf{e}h|_\mathbf{e} = o(r^{-\frac{d-2}{2}})$ for $k\leqslant 2$ as this condition ensures that the boundary terms vanish at infinity and one simply has to deal with the boundary $\mathbb{S}^{d-1}$.
\end{proof}

\begin{rem}\label{invariance qté}
    Just as in Proposition \ref{prop B T}, the left-hand side of \eqref{annulation T cas euclidien} has invariance properties on the hypersurface or action by diffeomorphism. The invariance by hypersurface is again a consequence of the divergence theorem, and the invariance by Lie derivative is comes from Proposition \ref{covariance reparam}.
\end{rem}

\begin{rem}
    The conformal Killing vector field $r\partial_r$ is present for any Ricci-flat cone. One can therefore define a similar integral quantity which has to vanish with the exact same proof. 
\end{rem}

\subsection{Deformations with constant scalar curvature}
In our situations of interest, as we will see later, we will consider Einstein deformations for which the second variation of the scalar curvature is constant. In this case, the obstruction becomes much simpler as we may drop the additional boundary term.
\begin{prop}\label{obst hk a scal2 nul}
    Let $h$ and $k$ be infinitesimal Einstein deformations of $\mathbf{e}$ on $B_\mathbf{e}(1+\epsilon)$ or $\mathbb{R}^d \backslash B_\mathbf{e}(1-\epsilon)$ satisfying
    \begin{enumerate}
        \item $\delta_\mathbf{e} h = \delta_\mathbf{e}k =0$
        \item $\E_\mathbf{e}^{(1)}(h) + \lambda \mathbf{e} =0$ and $\E_\mathbf{e}^{(1)}(k) + \mu \mathbf{e}=0$,
        \item and $\R_\mathbf{e}^{(2)}(h,k)$ is constant.
    \end{enumerate}
    Then, we have the obstruction:
    \begin{equation}
        \int_{\mathbb{S}^{d-1}} \big(\mathring{\Ric}_\mathbf{e}^{(2)}(h,k)\big)(r\partial_r,\partial_r)dv_{\mathbb{S}^{d-1}} = 0\label{obst killing hk}
    \end{equation}
    and for any Killing vector field $Y$
    \begin{equation}
        \int_{\mathbb{S}^{d-1}} \big(\mathring{\Ric}_\mathbf{e}^{(2)}(h,k)\big)(Y,\partial_r)dv_{\mathbb{S}^{d-1}} = 0\label{obst conform hk}
    \end{equation}
    where $\mathring{\Ric}$ is the traceless part of the Ricci curvature.
\end{prop}
\begin{rem}
    The first assumption is a gauge-fixing condition which we will always be able to assume up to acting by a diffeomorphism. Moreover it will always be satisfied if $h=k$ is an integrable Einstein deformation which is not Ricci-flat. Indeed, differentiating $\E(g_t)+\lambda(t) = 0$ starting at $g_0 = \mathbf{e}$ with $\partial_{t|t=0} {g_t} = h$, and $\partial^2_{t^2|t=0} {g_t} = h'$ twice yields $$\E_\mathbf{e}^{(1)}(h') + \E_\mathbf{e}^{(2)}(h,h) + \lambda''(0)\mathbf{e} + 2\lambda'(0)h = 0 $$
    where every term but maybe $2\lambda'(0) h = 2\lambda h$ is divergence-free by Lemma \ref{gauge properties Bianchi}.
\end{rem}
\begin{rem}
    The second assumption just means that we consider infinitesimal Einstein deformations.
    We will see how to ensure that the third assumption is satisfied in dimension $4$ in the next sections.
\end{rem}
\begin{proof}
For this situation, we use Schoen's Pohozaev identity \eqref{schoen poho} applied to the $2$-tensor $ \E_{\mathbf{e}}^{(2)}(h,k) $ which is divergence-free because $h$ and $k$ are infinitesimal Einstein deformations, see Lemma \ref{gauge properties Bianchi}:
\begin{equation}
\begin{aligned}
    \int_{\partial\Omega}\Big(\E_{\mathbf{e}}^{(2)}(h,k)\Big)^\circ(r\partial_r,n) dv_{{\mathbf{e}}_{|\partial\Omega}} &= -\frac{1}{2d}\int_\Omega \mathcal{L}_{r\partial_r} \tr_{\mathbf{e}} \Big(\E_{\mathbf{e}}^{(2)}(h,k)\Big) dv_{\mathbf{e}}.
\end{aligned}\label{resultat poho}
\end{equation}
Let us now express both $(\E_{\mathbf{e}}^{(2)}(h,k))^\circ$ and $\tr_{\mathbf{e}} (\E_{\mathbf{e}}^{(2)}(h,k))$ in terms of $\mathring{\Ric}^{(2)}(h,k)$ and $\R^{(2)}_\mathbf{e}(h,k)$. For this, we first note that since $\tr_g \E_g = \frac{2-d}{2}\R_g$ for any $g$, we have: 
\begin{equation}
    \tr_{\mathbf{e}} \Big(\E_{\mathbf{e}}^{(2)}(h,k)\Big) = \frac{2-d}{2}\R^{(2)}_\mathbf{e}(h,k) - \lambda \tr_\mathbf{e}k - \mu \tr_\mathbf{e}h, \label{R^2}
\end{equation} 
and then, by \eqref{traceless Ric2}, we find:
\begin{equation}
    \Big(\E_{\mathbf{e}}^{(2)}(h,k)\Big)^\circ  = \mathring{\Ric}^{(2)}(h,k) - \lambda \mathring{k} - \mu \mathring{h}. \label{traceless E^2}
\end{equation}
Finally, since by assumption $\delta_\mathbf{e}h = \delta_\mathbf{e}k =0$, we can again use \eqref{schoen poho}:
\begin{equation}
\begin{aligned}
    \int_{\partial\Omega}\mathring{h}(r\partial_r,n) dv_{{\mathbf{e}}_{|\partial\Omega}}= -\frac{1}{2d}\int_\Omega \mathcal{L}_{r\partial_r} (\tr_{\mathbf{e}}h) dv_{\mathbf{e}} \text{ and }\int_{\partial\Omega}\mathring{k}(r\partial_r,n) dv_{{\mathbf{e}}_{|\partial\Omega}}= -\frac{1}{2d}\int_\Omega \mathcal{L}_{r\partial_r} (\tr_{\mathbf{e}}k) dv_{\mathbf{e}}.
\end{aligned}\label{resultat poho h, k}
\end{equation}
Putting \eqref{resultat poho}, \eqref{R^2}, \eqref{traceless E^2} and \eqref{resultat poho h, k}, we find 
$$ \int_{\partial\Omega}\mathring{\Ric}^{(2)}(h,k)(r\partial_r,n) dv_{{\mathbf{e}}_{|\partial\Omega}} = -\frac{2-d}{4d}\int_\Omega \mathcal{L}_{r\partial_r} (\R^{(2)}_\mathbf{e}(h,k)) dv_{\mathbf{e}}. $$
and since $\R_\mathbf{e}^{(2)}(h,k)$ is assumed to be constant, we find the stated equality. The easier case of a Killing vector field $Y$ is treated similarly.
\end{proof}

\subsection{Einstein metrics closing-up inside a hypersurface}\label{closing up EH}
Let us consider a topology $ N = (\mathbb{R}^4\slash\Gamma)\backslash \{0\}\cup \Sigma $ for $\Sigma$ a lower dimensional submanifold ``closing up'' $N$ where $\{0\}$ should be. This should be the topology of a general Ricci-flat ALE, see \cite{andL2}. We see that $(\mathbb{R}^4\slash\Gamma,\mathbf{e})$ is isometric to a \emph{degenerate} metric on $N$ which we still denote $\mathbf{e}$ for which $\mathbf{e}_{|\Sigma}=0$ (meaning the induced metric on $\Sigma$, not the restriction of $\mathbf{e}$ at $\Sigma$), the vector field $r\partial_r$ also extends to $N$ and vanishes on $\Sigma$. 
    
    The typical situation is that of a minimal resolution of $\mathbb{C}^2\slash\Gamma$ for $\Gamma \subset SU(2)$, for instance: $T^*\mathbb{S}^2 =(\mathbb{R}^4\slash\mathbb{Z}_2)\backslash \{0\}\cup \mathbb{S}^2$ is the topology of the Eguchi-Hanson metric:
    \begin{equation}
    \mathbf{eh}:= \sqrt{\frac{r^4}{1+r^4}}(dr^2+r^2\alpha_1^2) + \sqrt{1+r^4}(\alpha_2^2+\alpha_3^2).\label{eguchihanson}
\end{equation}
with metric $ \alpha_2^2+\alpha_3^2 $ on $\mathbb{S}^2$.
Let us restrict ourselves to this situation and keep our discussion at a somewhat informal level as the last section of the article will prove these obstructions rigorously. Assume that there exists a smooth curve of Einstein metrics $t\in [0,1]\mapsto \mathbf{g}_t$ on $N = T^*\mathbb{S}^2$ with $ \mathbf{g}_t $ nondegenerate for $t>0$ and with $\mathbf{g}_0 =\mathbf{e}$. This implies by \cite{and,bkn,nak} that some rescaling of $\mathbf{g}_t$ converges to the Eguchi-Hanson metric \eqref{eguchihanson}
    where, schematically, one has \emph{in the coordinates} of \eqref{eguchihanson}:
    $$\mathbf{eh} = \mathbf{e} + H^4 +\ldots$$ with $|H^4|_\mathbf{e} \sim r^{-4}$, while another rescaling of $\mathbf{g}_t$ converges to an Einstein orbifold $$\mathbf{g}_o = \mathbf{e} +H_2 +\ldots$$ with $|H_2|_\mathbf{e} \sim r^{2}$, and with singularity $\mathbb{R}^4\slash\mathbb{Z}_2$. Up to some gauge conditions, we are therefore in the situation of Theorem \ref{obstruction desing taub's} but we may use the integration by parts of Proposition \ref{obst hk a scal2 nul} to interpret this as an obstruction to the existence of an Einstein metric ``closing-up'' inside the hypersurface $\mathbb{S}^3\slash\mathbb{Z}_2$. 
    
    By \cite{ozu3}, up to rescaling and reparametrizing the curve $t\mapsto\mathbf{g}_t$ in well-chosen coordinates, we have a development $\mathbf{g}_t = \mathbf{e} + H_2 + t^2H^4 + t^2H^4_2+...$ with $H^4_2 = \mathcal{O}(r^{-2})$ with remaining term negligible in a region where $\sqrt{t}\ll r\ll 1$, hence in particular for $r$ close to $t^{1/4}$ for small $t$. We again consider the integration by parts
        \begin{align}
    \int_{\{r=t^{1/4}\}}& \Big(\E^{(1)}(H^4_2)+\E^{(2)}_\mathbf{e}(H_2,H^4)+\lambda H^4\Big)(r\partial_r,\partial_r) dv_{t^{1/4}\mathbb{S}^3\slash\mathbb{Z}_2}\nonumber\\
    &= \int_{\{r<t^{1/4}\}} \tr_\mathbf{e}\Big(\E^{(1)}(H^4_2)+\E^{(2)}_\mathbf{e}(H_2,H^4)+\lambda H^4\Big) dv_\mathbf{e}.\label{dvp eh desing}
\end{align}
It was proven in \cite{ozu3} that in the coordinates of \eqref{eguchihanson}, one actually has $H^4_2=0$. This yields: 
\begin{align*}
    \int_{\{r=t^{1/4}\}}& \Big(\mathring{\Ric}^{(2)}_\mathbf{e}(H_2,H^4)\Big)(r\partial_r,\partial_r) dv_{t^{1/4}\mathbb{S}^3\slash\mathbb{Z}_2}
    = \frac{2-d}{2}\int_{\{r<t^{1/4}\}} \R^{(2)}_\mathbf{e}(H_2,H^4) dv_\mathbf{e}.
\end{align*}
Now, by Corollary \ref{courbure si lun est seldual} proven below, one has $\R^{(2)}_\mathbf{e}(H_2,H^4) = 0$ because $H^4$ is anti-selfdual, and we recover the obstruction \eqref{obst killing hk} in this situation. The other obstructions with Killing vector fields are recovered in the same way.

\begin{rem}
    The coordinates of \eqref{eguchihanson} correspond to the \emph{volume gauge} of Definition \ref{def volume gauge} below and one of their properties is that the vector field $r\partial_r$ is harmonic, hence $\mathcal{L}_{r\partial_r}\mathbf{eh}$ is an infinitesimal Einstein deformation -- up to a trace term.
\end{rem}

\section{The $4$-dimensional situation}\label{section 4d}

Let us now specialize our discussion to the dimension $d=4$ where most of our applications are. In this section, we will test the obstruction to the integrability of infinitesimal Einstein deformations on the development of Einstein $4$-manifolds at the infinity of a Ricci-flat ALE metric. See Appendix \ref{appendix mf} for the case of a neighborhood of a given point of an Einstein manifold or orbifold. We will see that the obstructions \emph{always} vanish and do not add any restriction.

In order to show this, we will prove the existence of good gauges in which the quadratic term of $h\mapsto \Ric(\mathbf{e}+h)$ are easily computable. This is an important step towards the next Section \ref{section obst desing} where obstructions to the desingularization of some Einstein metrics are found. 

\subsection{Notations}

In dimension $4$, the space of $2$-forms decomposes into selfdual and anti-selfdual $2$-forms which are elements of the eigenspaces of Hodge star operator $*$ (which satisfies $*^2 = \textup{Id}$) respectively associated to the eigenvalues $1$ and $-1$. We denote $\Omega^+$ and $\Omega^-$ the associated eigenspaces.

Denote $(x_1,x_2,x_3,x_4)$ coordinates in an orthonormal basis of $\mathbb{R}^4$. We define the $2$-forms $$\omega^\pm_1 := dx^1\wedge dx^2\pm dx^3\wedge dx^4$$ and similarly $\omega_2^\pm$ and $\omega_3^\pm$ by cyclic permutations. The $\omega_i^+$ form an orthogonal basis of the space of selfdual $2$-forms, $\Omega^+$, and the $\omega_i^-$ form an orthogonal basis of the space of anti-selfdual $2$-forms, $\Omega^-$. 

Thanks to them, we define the following basis of the Killing vector fields preserving $0$ on $\mathbb{R}^4$:
\begin{equation}
   Y_i^\pm:= \omega_i^\pm(r\partial_r).
\end{equation}
 The other Killing vector fields of $\mathbb{R}^4$ are given by constant vector fields representing translations. Note that those will not be $\Gamma$-invariant for any $\Gamma\subset SO(4)$ with $\Gamma\neq\{\Id\}$.

Note that in each orientation, the frame $(r\partial_r,Y_1^\pm,Y_2^\pm,Y_3^\pm) $ is dual to the coframe $ (dr/r,\alpha_1^\pm,\alpha_2^\pm,\alpha_3^\pm ) $ where we define $\alpha_i^\pm:=\omega_i^\pm(dr/r)$. 
We also define the following $2$-forms which also form bases of the spaces of selfdual or anti-selfdual $2$-forms:
$$\theta_1^\mp :=  r dr\wedge \alpha_1^\pm \mp r^2\alpha_2^\pm\wedge\alpha_3^\pm,$$ and similarly $\theta_2^\mp$ and $\theta_3^\mp$ by cyclic permutations.
\begin{rem}
    The notation $\theta_i^+$ meant something else in \cite{ozu3}. 
\end{rem}
 For the above $2$-forms, we have the following formula: for $x = (x_1,x_2,x_3,x_4)$
    \begin{align}
        \theta_i^\mp(x) = \sum_{j=1}^3 \frac{x^T(\omega_i^\pm\circ\omega_j^\mp)x}{|x|^2}\omega_j^\mp =- \frac{\langle Y_i^\pm, Y_j^\mp\rangle}{r^2}\omega_j^\mp,\label{rotation 2 forms -}
    \end{align}
    where $ \omega_i^\pm\circ\omega_j^\mp $ is the symmetric traceless matrix given by the (commuting) product of the antisymmetric matrices associated to $\omega_i^\pm$ and $\omega_j^\mp$, and where $x^T$ is the transpose of $x$.
We also have the following equalities: $\omega_1^\pm = r dr\wedge \alpha_1^\pm \pm r^2\alpha_2^\pm\wedge\alpha_3^\pm$
    and similar equalities for $\omega_i^\pm$ for $i \in \{2,3\}$ by cyclic permutations.

\subsubsection{Orbifolds and ALE spaces}

We will be interested in two types of geometries: Einstein orbifolds and Ricci-flat ALE metrics. They respectively correspond to the singular limits and the singularity models of the degeneration of Einstein $4$-manifolds.

\begin{defn}[Orbifold (with isolated singularities)]\label{orb Ein}
    We will say that a metric space $(M_o,g_o)$ is an orbifold of dimension $d\geqslant 2$ if there exists $\epsilon_0>0$ and a finite number of points $(p_k)_k$ of $M_o$ called \emph{singular} such that we have the following properties:
    \begin{enumerate}
        \item the space $(M_o\backslash\{p_k\}_k,g_o)$ is a manifold of dimension $d$,
        \item for each singular point $p_k$ of $M_o$, there exists a neighborhood of $p_k$, $ U_k\subset M_o$, a finite subgroup acting freely on $\mathbb{S}^{d-1}$, $\Gamma_k\subset SO(n)$, and a diffeomorphism $ \Phi_k: B_\mathbf{e}(0,\epsilon_0)\subset\mathbb{R}^d\slash\Gamma_k \to U_k\subset M_o $ for which, the pull-back of $\Phi_k^*g_o$  on the covering $\mathbb{R}^d$ is smooth.
    \end{enumerate}
\end{defn}

\begin{rem}
    Note that smooth Einstein metrics are Einstein orbifolds. Einstein orbifold metrics are smooth up to taking a finite local cover at the singular point as seen in \cite{bkn}.
\end{rem}

\begin{defn}[ALE orbifold (with isolated singularities)]\label{def orb ale}
    An ALE orbifold of dimension $d\geqslant 4$, $(N,g_b)$ is a metric space for which there exists $\epsilon_0>0$, singular points $(p_k)_k$ and a compact $K\subset N$ for which we have:
    \begin{enumerate}
        \item $(N,g_b)$ is an orbifold of dimension $d$,
        \item there exists a diffeomorphism $\Psi_\infty: (\mathbb{R}^d\slash\Gamma_\infty)\backslash \overline{B_\mathbf{e}(0,\epsilon_0^{-1})} \to N\backslash K$ such that we have $r^l|\nabla^l(\Psi_\infty^* g_b - \mathbf{e})|_{\mathbf{e}}\leqslant C_l r^{-d}.$
    \end{enumerate}
\end{defn}

\begin{note}
    We will often identify $\mathbb{R}^d\slash\Gamma$ and its cover $\mathbb{R}^d$ when writing ALE of orbifold spaces in coordinates.
\end{note}

\subsubsection{Curvature of Einstein $4$-manifolds}

Thanks to the direct sum of selfdual and anti-selfdual $2$-forms, the symmetric endomorphism on $2$-forms, $\mathbf{R}$ given by the Riemannian curvature decomposes into blocks,
\[
\mathbf{R}=:
  \begin{bmatrix}
     \mathbf{R}^+ & \mathring{\Ric} \\
     \mathring{\Ric} & \mathbf{R}^-
  \end{bmatrix},
\]
 where the $\mathring{\Ric}$ is the traceless part of the Ricci curvature, and where $\mathbf{R}^\pm$ are the selfdual and anti-selfdual parts of the curvature.

Let us do a quick recap on the curvature of Einstein perturbations of the Euclidean space $(\mathbb{R}^4,\mathbf{e})$. The starting point is the identification of the set of traceless symmetric $2$-tensors $\textup{Sym}_0^2TM$ with $\Omega^+\otimes\Omega^-$ thanks to the map:
$$\omega^+\otimes\omega^-\in\Omega^+\otimes\Omega^- \mapsto \omega^+\circ \omega^- =\omega^-\circ \omega^+\in\textup{Sym}_0^2TM$$
where $\omega^+\circ \omega^-$ is the $2$-tensor associated to the composition of the anti-symmetric endomorphisms of $TM$ associated to $\omega^+$ and $\omega^-$ by the metric. Therefore, any $2$-tensor $h$ on $(\mathbb{R}^4,\mathbf{e})$ has unique decompositions:
$$h = \lambda \mathbf{e} + \sum_i \phi_i^-\circ\omega_i^+ =\lambda \mathbf{e} + \sum_j \phi_j^+\circ\omega_j^-$$
for a scalar function $\lambda$, and $\phi_i^-\in \Omega^-$ and $\phi_j^+\in \Omega^+$. According to \cite{biq2}, the Bianchi gauge condition for $h$ rewrites:
\begin{equation}
    d\lambda + \sum_{i= 1}^3*\left(\omega_i^+\wedge(*d\phi_i^-)\right) = d\lambda + \sum_{j= 1}^3*\left(\omega_j^-\wedge(*d\phi_j^+)\right) = 0. \label{biachi gauge}
\end{equation}
Extending the computations of \cite{biq2}, we prove the following result.
\begin{prop}
    Let $h$ be a symmetric $2$-tensor on $\mathbb{R}^4$ decomposed as:
    $$h = \lambda \mathbf{e} + \sum_i \phi_i^-\circ\omega_i^+ =\lambda \mathbf{e} + \sum_j \phi_j^+\circ\omega_j^-$$
    and satisfying the condition \eqref{biachi gauge}.
    
    Then, defining $a^{+,(1)}_\mathbf{e}(h) = \sum_i*d\phi_i^-\otimes\omega_i^+$ and $a^{-,(1)}_\mathbf{e}(h) = \sum_j*d\phi_j^+\otimes\omega_j^-$ the infinitesimal variations of connections, we have:
    \begin{itemize}
        \item $\mathbf{R}^{+,(1)}_\mathbf{e}(h) = -d_+a^{+,(1)}_\mathbf{e}(h) = -\sum_{i} d_+*d\phi_i^-\otimes \omega_i^+$ where $d_+$ is the exterior differential composed with the projection on $\Omega^+$,
        \item $\mathbf{R}^{-,(1)}_\mathbf{e}(h) = -d_-a^{-,(1)}_\mathbf{e}(h) = -\sum_{i} d_-*d\phi_i^+\otimes \omega_i^-$ where $d_-$ is the exterior differential composed with the projection on $\Omega^-$, and
        \item $\mathring{\Ric}^{(1)}_\mathbf{e}(h) = d_-a^{+,(1)}_\mathbf{e}(h)= d_+a^{-,(1)}_\mathbf{e}(h) $ and in particular, $h$ is an infinitesimal Einstein deformation if and only if for all $i\in\{1,2,3\}$,
        $d_-*d\phi_i^-=0$
        or equivalently if and only if for all $j\in\{1,2,3\}$,
        $d_+*d\phi_j^+=0.$
    \end{itemize}
    Moreover, if $h$ is an infinitesimal Einstein deformation, then we have:
    \begin{itemize}
        \item $\mathbf{R}^{+,(2)}_\mathbf{e}(h,h) =- \frac{1}{2}\big[a_\mathbf{e}^{+,(1)}(h),a_\mathbf{e}^{+,(1)}(h)\big]_+$,
        \item $\mathbf{R}^{-,(2)}_\mathbf{e}(h,h) =- \frac{1}{2}\big[a_\mathbf{e}^{-,(1)}(h),a_\mathbf{e}^{-,(1)}(h)\big]_-$ and
        \item defining linear maps $\phi^- : \Omega^+\to \Omega^-$ and $\phi^+ : \Omega^-\to \Omega^+$ by for any $i$ and $j$ in $\{1,2,3\}$, $ \phi^-(\omega_i^+) = \phi_i^- $ and $ \phi^+(\omega_j^-) = \phi_j^+ $, we have
        \begin{align*}
            \mathring{\Ric}^{(2)}_\mathbf{e}(h,h) &= \frac{1}{2}\big[a_\mathbf{e}^{+,(1)}(h),a_\mathbf{e}^{+,(1)}(h)\big]_- +\sum_i\phi^-\big(\mathbf{R}^{+,(1)}_\mathbf{e}(h)(\omega_i^+)\big)\otimes \omega_i^+\\ &=\frac{1}{2}\big[a_\mathbf{e}^{-,(1)}(h),a_\mathbf{e}^{-,(1)}(h)\big]_+ +\sum_j\phi^+\big(\mathbf{R}^{-,(1)}_\mathbf{e}(h)(\omega_j^-)\big)\otimes \omega_j^-.
        \end{align*}
    \end{itemize}
\end{prop}
\begin{proof}
    The only points that are not proven in \cite{biq2} are the values of $\mathbf{R}^{\pm,(2)}_\mathbf{e}(h,h)$. For this, we recall that $\mathbf{R}^+$ is the opposite of the curvature $d a^+ + \frac{1}{2}\big[a^+,a^+\big]$ of the bundle $\Omega^+$. The origin of the term $- \frac{1}{2}\big[a_\mathbf{e}^{+,(1)}(h),a_\mathbf{e}^{+,(1)}(h)\big]_+$ is therefore clear. There is another source of variation of $\mathbf{R}^+$ which are the variations of $\Omega^+$ and $\Omega^-$ at $\mathbf{e}$ in the direction $h = \lambda \mathbf{e} + \sum_i \phi_i^-\circ\omega_i^+=\lambda \mathbf{e} + \sum_j \phi_j^+\circ\omega_j^-$. According to \cite{biq2}, by blocks given by the direct sum $\Omega^+_\mathbf{e}\oplus \Omega^-_\mathbf{e}$, the variation of the bases $(\omega_i^\pm)_i$ as bases of $\Omega^\pm$ rewrites:
    \begin{equation}
        \omega_{i,\mathbf{e}}^{\pm,(1)}(h)=\begin{bmatrix}
           \lambda & -\phi^{-}\\
           -\phi^+& \lambda
        \end{bmatrix}\omega_i^+,
    \end{equation}
    and therefore, the conjugation by $\textup{Id}+\begin{bmatrix}
           \lambda & -\phi^{-}\\
           -\phi^+& \lambda
        \end{bmatrix} +\mathcal{O}(|h|_\mathbf{e}^2)$ of the first order curvature variation $
        \begin{bmatrix}
           \mathbf{R}^{+,(1)}_\mathbf{e}(h)&0\\
           0&\mathbf{R}^{-,(1)}_\mathbf{e}(h)
        \end{bmatrix}$
        leaves at second order term:
        $
            \begin{bmatrix}
              0 &\phi^{-}(\mathbf{R}^{+,(1)}_\mathbf{e}(h))\\
           \phi^{+}(\mathbf{R}^{-,(1)}_\mathbf{e}(h))&0
            \end{bmatrix},$ which yields the result.
\end{proof}

\subsubsection{An application to the variation of scalar curvature.}

    Let $h = \lambda \mathbf{e} + \sum_i \phi_i^-\circ\omega_i^+$ and $k=\mu \mathbf{e} + \sum_i \psi_i^-\circ\omega_i^+$ for $\lambda$ and $\mu$ scalar functions and $\phi_i^-,\psi_i^-\in\Omega_-$. 
\begin{cor}\label{courbure si lun est seldual}
Assume that $h$ and $k$ are infinitesimal Einstein deformations satisfying \eqref{biachi gauge} and assume moreover that $h$ is anti-selfdual in the sense that $a^{+,(1)}_\mathbf{e}(h)=0$. 
    
    Then, we have: $\mathbf{R}_\mathbf{e}^{+,(2)}(h,k)=0,$
    and moreover:
    \begin{equation}
        \mathring{\Ric}_\mathbf{e}^{(2)}(h,k)= \sum_i \phi_i^-\big(\mathbf{R}_\mathbf{e}^{+,(1)}(k)(\omega_i^+)\big)\otimes\omega_i^+.\label{Ric2 si scal cst}
    \end{equation}
\end{cor}

\subsection{Asymptotic curvature of $4$-dimensional Ricci-flat ALE metrics}

Let us now study the infinity of Ricci-flat ALE metrics of dimension $4$.

\subsubsection{Development of Ricci-flat ALE metrics}
Let us start by understanding the asymptotic terms of Ricci-flat ALE metrics.

Let $ (N,\mathbf{g}_b) $ be a $4$-dimensional Ricci-flat ALE orbifold asymptotic to $ (\mathbb{R}^4\slash\Gamma,\mathbf{e}) $ for $\Gamma\subset SO(4)$, and let $(\Sigma_s)_{s>s_0}$ for some $s_0>0$ be a CMC (Constant Mean Curvature) foliation of $ (N,\mathbf{g}_b) $ in a neighborhood of infinity as in \cite{ce,bh} where the mean curvature of $\Sigma_s$ is equal to $\frac{3}{s}$ (like a sphere of radius $s$ in $(\mathbb{R}^4\slash\Gamma,\mathbf{e})$). 
\begin{defn}[CMC gauge]\label{def cmc gauge}
Then, by \cite{bh} there exist a compact $K\subset N$, $s_0>0$ and a local diffeomorphism: $\Phi: (\mathbb{R}^4\backslash B_\mathbf{e}(0,s_0))\slash\Gamma\mapsto N\backslash K $ with:
\begin{itemize}
    \item for all $s>s_0$,
    $\Phi(S_\mathbf{e}(s)) = \Sigma_s,$
    \item $\Phi^*\mathbf{g}_b-\mathbf{e} = H^4 + \mathcal{O}(r^{-5})$ at infinity for $|H^4|_\mathbf{e}\sim r^{-4}$,
    \item $ \delta_\mathbf{e}H^4 = 0 $, $\tr_\mathbf{e} H^4 = 0$, $H^4(\partial_r,\partial_r) = 0$
    \item and one can even arrange $\Phi$ so that $H^4$ is more precisely of the form:
$$H^4 = \frac{\sum_{ij}h_{ij}^+\theta_i^-\circ\omega_j^+ + \sum_{kl}h_{kl}^-\theta_k^+\circ\omega_l^-}{r^4}.$$
with $\sum_{i}h_{ii}^+=0$ and $ \sum_{k}h_{kk}^- = 0$.

We call these coordinates a \emph{CMC gauge}.
\end{itemize}
\end{defn}
\begin{rem}
    Instead of that last point, in \cite{bh}, another decomposition is used with so-called \emph{reduced Kronheimer's terms} which are projections of the terms $\theta_i^\mp\circ\omega_j^\pm$ on their part without $dr$ using some \emph{gauge terms} of the form $\mathcal{L}_V$ for $V = \frac{1}{r^4}L(x)$ for $L$ either proportional to the identity or skew-symmetric.
\end{rem}
Denote $\Omega_s$ the interior of the hypersurface $\Sigma_s$. The limit:
\begin{equation}
   \mathcal{V}(N,\mathbf{g}_b):=\lim_{s\to \infty } \vol_{\mathbf{g}_b}\Omega_s - \vol_\mathbf{e}(B_\mathbf{e}(s)\slash\Gamma)\leqslant 0 \label{reduced volume}
\end{equation}
exists and is called the \emph{renormalized volume} of $(N,\mathbf{g}_b)$ in \cite{bh}. It is vanishing \emph{if and only if} $(N,\mathbf{g}_b)$ is flat.
\begin{rem}
    An interesting remark of Hans-Joachim Hein is the following. There are examples of Kronheimer's instantons of \cite{kro} described in Remark \ref{rem kro gauge} below for which one has $H^4=0$ in CMC gauge. It is a way to see that the notion of reduced volume is \emph{not} an asymptotic quantity. It essentially does not tell anything about the asymptotics of Ricci-flat ALE metric at infinity which at this point could vanish at any order. It is a global and subtle quantity. This reduced volume is the core quantity of the obstruction \eqref{premiere equation rdr} proven later.
\end{rem}

We will also need a particular gauge very close to one introduced by Biquard and Hein in some unpublished notes which led to \cite{bh}. This is a so-called \emph{volume gauge} in which the volume form of the ALE metric is asymptotically equal to that of the asymptotic flat cone.
\begin{exmp}\label{rem eguchihanson}
    The volume form of the Eguchi-Hanson metric in its usual form \eqref{eguchihanson} is equal to that of the asymptotic cone $\mathbb{R}^4\slash\mathbb{Z}_2$ with its flat metric $dr^2 + r^2 (\alpha_1^2+\alpha_2^2+\alpha_3^2)$.
\end{exmp}

Precomposing the above local diffeomorphism $\Phi$ of Definition \ref{def cmc gauge} with the flow of $r^{-3}\partial_r$ for an adapted amount of time $t=C(\Gamma)\mathcal{V}(N,\mathbf{g}_b)$ for $C(\Gamma)>0$, one may choose another diffeomorphism $\Psi$ between neighborhoods of infinities with the following properties.
\begin{defn}[Volume gauge]\label{def volume gauge}
There exist a compact $K'\subset N$, $r'_0>0$ and a local diffeomorphism: $\Psi: (\mathbb{R}^4\slash\Gamma)\backslash B_\mathbf{e}(0,r'_0)\mapsto N\backslash K' $ with:
\begin{itemize}
    \item $\Psi^*\mathbf{g}_b-\mathbf{e} = H^4 + \mathcal{O}(r^{-5})$ at infinity for $|H^4|_\mathbf{e}\sim r^{-4}$,
    \item Denoting $\Omega'_s$ the interior of $\Psi^*(s\mathbb{S}^3\slash\Gamma)$, $$\lim_{s\to \infty } \vol_{\mathbf{g}_b}\Omega_s' - \vol_\mathbf{e}(B_\mathbf{e}(s)\slash\Gamma)=0, $$
    \item $ \delta_\mathbf{e}H^4 = 0 $, $\tr_\mathbf{e} H^4 = 0$, $H^4(\partial_r,\partial_r) =c(\Gamma)\mathcal{V}(N,\mathbf{g}_b) $ for $c(\Gamma)>0$.
\item more precisely, 
$$H^4 = \frac{\sum_{ij}h_{ij}^+\theta_i^-\circ\omega_j^+ + \sum_{kl}h_{kl}^-\theta_k^+\circ\omega_l^-}{r^4}.$$
with $\sum_{i}h_{ii}^+ + \sum_{k}h_{kk}^- = c(\Gamma)\mathcal{V}(N,\mathbf{g}_b)\leqslant 0$ as in the previous point.
\end{itemize}
\end{defn}

As in Proposition \ref{terme quadratique orbifold} in the Appendix, we see the term $H^4$ in this volume gauge as only determined by the curvature at infinity and the reduced volume.
\begin{prop}\label{asympt ALE jauge volume}
    Let $(N,\mathbf{g}_b)$ be a non flat Ricci-flat ALE orbifold asymptotic to $\mathbb{R}^4\slash\Gamma$ for $\Gamma \subset SO(4)$ with reduced volume $\mathcal{V}(N,\mathbf{g}_b)<0$.

    Then, in volume gauge as in Definition \ref{def volume gauge}, up to changing the bases $(\omega_i^+)_i$ and $(\omega_k^-)_k$ to diagonalize the curvature, one has the asymptotic $\mathbf{g}_b = \mathbf{e} + H^4+\mathcal{O}(r^{-5})$ with \begin{equation}
        H^4 = -\frac{\sum_{i}h_{ii}^+\theta_i^-\circ\omega_i^+ + \sum_{k}h_{kk}^-\theta_k^+\circ\omega_lk-}{r^4}
    \end{equation}
    with $\sum_i h_{ii}^+ + \sum_k h_{kk}^- = - c(\Gamma)\mathcal{V}(N,\mathbf{g}_b)>0$ and the induced asymptotic curvature in the basis $(\theta_i^\pm)_i$ satisfies 
    $$\mathbf{R}^{\mp,(1)}_\mathbf{e}(H^4) = \frac{8}{r^6}\begin{bmatrix}
       2h^\pm_{11} -h^\pm_{22}-h^\pm_{33}&0&0\\
       0&-h^\pm_{11} +2h^\pm_{22}-h^\pm_{33}&0\\
       0&0&-h^\pm_{11} -h^\pm_{22}+2h^\pm_{33}
    \end{bmatrix}.$$
\end{prop}
\begin{rem}\label{asymptotic gauge term kro}
    It is possible that the curvature of $(N,\mathbf{g}_b)$ is decaying faster than $r^{-6}$. In this case, the $H^4$ term is purely a gauge term of the form: 
    $$ \frac{1}{r^4}\sum_i\theta_i^-\circ\omega_i^+ =\frac{1}{r^4}\sum_k\theta_k^+\circ\omega_k^- = \frac{3 dr^2 - r^2g_{\mathbb{S}^3}}{r^4} = \frac{1}{2}\Hess_\mathbf{e}(r^{-2}).$$
    See Remark \ref{rem kro gauge} below for an example.
\end{rem}
\begin{proof}
    We first check that 
    $H^4_+:=-\sum_{i}h_{ii}^+\frac{\theta_i^-\circ\omega_i^+}{r^4}$
    only induces anti-selfdual curvature at the first order. The linearization at $\mathbf{e}$ of the induced connection on $\Omega^+$ is actually zero since $d(\frac{\theta_i^+}{r^4})=0$ and therefore the induced selfdual curvature vanishes as well. This also shows that the term is in Bianchi gauge.
    
    In order to compute the induced curvature, one uses Kronheimer's examples. In particular the curvature of the Eguchi-Hanson metric which with $H^4 = -\frac{\theta_1^-\circ\omega_1^+}{2r^4}$ induces the curvature
    $$\mathbf{R}^{\mp,(1)}_\mathbf{e}(H^4) = \frac{4}{r^6}\begin{bmatrix}
       2&0&0\\
       0&-1&0\\
       0&0&-1
    \end{bmatrix}.$$
    by \cite{eh}. By linearity, one attains any $H^4$ as above, and the result follows.
\end{proof}

In the CMC gauge of Definition \ref{def cmc gauge}, we find a simpler expression.
\begin{cor}\label{courbure coordonnées CMC}
    Let $(N,\mathbf{g}_b)$ be a Ricci-flat ALE orbifold asymptotic to $\mathbb{R}^4\slash\Gamma$ for $\Gamma \subset SO(4)$. Then, in CMC gauge as in Definition \ref{def cmc gauge}, up to changing the bases $(\omega_i^+)_i$ and $(\omega_k^-)_k$ to diagonalize the curvature, one has $\mathbf{g}_b = \mathbf{e} + H^4+\mathcal{O}(r^{-5})$ with \begin{equation}
        H^4 = -\frac{\sum_{i}h_{ii}^+\theta_i^-\circ\omega_i^+ + \sum_{k}h_{kk}^-\theta_k^+\circ\omega_lk-}{r^4}
    \end{equation}
    with $\sum_i h_{ii}^\pm=0$ and the induced asymptotic curvature satisfies in the basis $(\theta_i^\pm)_i$
    $$\mathbf{R}^{\mp,(1)}_\mathbf{e}(H^4) = \frac{24}{r^6}\begin{bmatrix}
       h^\pm_{11}&0&0\\
       0& h^\pm_{22}&0\\
       0&0&h^\pm_{33}
    \end{bmatrix}.$$
\end{cor}

\subsubsection{(Anti-)Selfdual Einstein deformations}

\begin{prop}
    If $H^4$, the $r^{-4}$-asymptotic term of a Ricci-flat ALE metric only induces an anti-selfdual (or selfdual) curvature, then it satisfies the obstructions \eqref{obst taub} and \eqref{annulation T cas euclidien}. 
\end{prop}
\begin{proof}
    By \cite{bh} or Proposition \ref{asympt ALE jauge volume}, up to a gauge term, we can assume that $H^4$ is a linear combination of $\frac{\theta_i^-\circ\omega_i^+}{r^4}$ terms.
    
    Thanks to Kronheimer's examples, we can reach any term of the form $-\sum_ih_{ii}^+\frac{\theta_i^-\omega_i^+}{r^4}$ with $h_{ii}^+\geqslant 0$. This condition $h_{ii}^+\geqslant 0$ can always be arranged up to adding a gauge term as in Remark \ref{asymptotic gauge term kro}. Since gauge terms do not matter in satisfying the obstructions \eqref{obst taub} and \eqref{annulation T cas euclidien} by Proposition \ref{prop B T}, we obtain the result.
\end{proof}

\subsubsection{General deformations}

\begin{prop}
    Let $(N,\mathbf{g}_b)$ be a Ricci-flat ALE metric whose asymptotic term in CMC gauge is: $H^4_++H^4_-$ where
    $$H_{+}^4 := -\sum_i h_{ii}^+\frac{\theta_i^-\circ\omega_i^+}{r^4} \;\text{ and }\; H_-^4:=-\sum_k h_{kk}^-\frac{\theta_k^+\circ\omega_k^-}{r^4}.$$
    Then, the obstructions \eqref{obst taub einstein} and \eqref{annulation T cas euclidien} are satisfied.
\end{prop}
\begin{proof}
    Let us consider $H_{+}^4 := -\sum_i h_{ii}^+\frac{\theta_i^-\circ\omega_i^+}{r^4}$ and $H_-^4:=-\sum_k h_{kk}^-\frac{\theta_k^+\circ\omega_k^-}{r^4}$ satisfying $\sum_i h_{ii}^+=\sum_k h_{kk}^-=0$. Recall that these terms only induce a nonvanishing connection respectively in the anti-selfdual and selfdual orientation and the formula $$\theta_i^\pm =\sum_k \langle \omega_i^\mp(\partial_r),\omega_l^\pm(\partial_r)\rangle \omega_l^\pm.$$ From the formula \eqref{Ric2 si scal cst} and Corollary \ref{courbure coordonnées CMC}, we therefore find:
    \begin{align*}
        \mathring{\Ric}_\mathbf{e}^{(2)}(H_{+}^4,H_{-}^4) 
        &= 24\sum_i\sum_kh_{ii}^+h_{kk}^-\frac{\langle \omega_k^-(\partial_r),\omega_i^+(\partial_r)\rangle\theta_i^-\circ \theta_k^+}{r^{10}}
    \end{align*}
    
    Now, using the fact that $\theta_i^\mp(\partial_r) = \omega_i^\mp(\partial_r)$, we see that the obstruction of Proposition \ref{obst hk a scal2 nul} is proportional to 
    $$ \sum_{i,k}h_{ii}^+h_{kk}^- \int_{\mathbb{S}^3} \langle \omega_k^-(\partial_r),\omega_i^+(\partial_r)\rangle_\mathbf{e}^2 dv_{\mathbb{S}^3}.$$
    Since both $\sum_i h_{ii}^+= 0$ and $\sum_kh_{kk}^- = 0$, the obstruction \eqref{obst conform hk} is satisfied. The proof is similar for the other obstructions \eqref{obst killing hk}.
\end{proof}

\section{Obstruction to the desingularization of Einstein metrics}\label{section obst desing}
We finally recover the obstructions of \cite{biq1} to the desingularization of Einstein metrics and many of the additional obstructions of \cite{ozu2} (but not the higher order obstructions of \cite{ozu3}). We conclude this section by proving that some Einstein orbifolds cannot be desingularized by smooth Einstein manifolds.

\subsection{Infinitesimal Einstein deformations of Ricci-flat ALE metrics}\label{section deformation ricci plat ale}
On $(\mathbb{R}^4\slash\Gamma,g_\mathbf{e})$, the vector field $ r\partial_{r}$ is a conformal Killing vector field. It is moreover half of the gradient of the function $u:=r^2$ which is a solution to $-\nabla_\mathbf{e}^*\nabla_\mathbf{e} u = 8$, and we have $\frac{1}{2}\mathcal{L}_{\nabla_\mathbf{e} u}\mathbf{e} = \textup{Hess}_{\mathbf{e}}u = 2\mathbf{e}$. On a Ricci-flat ALE orbifold we can extend this situation on the whole space as follows.

\begin{prop}[{\cite{bh,ozu2}}]\label{def by scaling}
    Let $(N,\mathbf{g}_b)$ be a Ricci-flat ALE orbifold asymptotic to $\mathbb{R}^4\slash\Gamma$. Then, there exists a unique vector field $X$ on $(N,\mathbf{g}_b)$ such that $\Phi^{*}X = r\partial_{r} + o(r)$, and $\nabla_{\mathbf{g}_b}^*\nabla_{\mathbf{g}_b} X =0$. We actually have $X=\frac{1}{2}\nabla_{\mathbf{g}_b} u$, where $u$ is the unique solution of $-\nabla_{\mathbf{g}_b}^*\nabla_{\mathbf{g}_b} u = 8$, such that $u = r^2 + o(1).$ Moreover, $(\mathcal{L}_X\mathbf{g}_b)^\circ = \mathcal{L}_X\mathbf{g}_b-2\mathbf{g}_b$, the traceless part of $\mathcal{L}_X\mathbf{g}_b$ is an infinitesimal Ricci-flat deformation of $\mathbf{g}_b$ which is trace-free and divergence-free.
\end{prop}
\begin{proof}[]
    
\end{proof}
\begin{prop}[{\cite[Section 4]{bh}}]\label{def par rotation}
        Let $(N,\mathbf{g}_b)$ be a Ricci-flat ALE orbifold asymptotic to $\mathbb{R}^4\slash\Gamma$. Then for any Killing vector field $Y$ on $\mathbb{R}^4\slash\Gamma$, there exists a unique vector field $Y'$ on $(N,\mathbf{g}_b)$ such that $Y' = Y + o(r).$ Moreover the infinitesimal deformation $\mathcal{L}_{Y'}\mathbf{g}_b$ is divergence-free and trace-free. It vanishes if and only if $Y'$ is a Killing vector field of $\mathbf{g}_b$.
\end{prop}

\begin{proof}[]
    
\end{proof}
\begin{rem}
    \emph{All} of the infinitesimal Einstein deformations of the Eguchi-Hanson metric \eqref{eguchihanson} are of the above types.
\end{rem}

\subsection{Obstructions to the desingularization}

Let us now show that we recover the obstructions of \cite{biq1,ozu2} and find another expression for them. This new expression will further highlight the link between these obstructions and the lack of integrability of Einstein desingularizations of \cite[Chapter 4]{ozuthese}.

\begin{lem}\label{lem volume gauge et asymptotiques o_i}
    Let $(N,\mathbf{g}_b)$ be a Ricci-flat ALE metric in volume gauge as in Definition \ref{def volume gauge}. And consider its development at infinity in this gauge:
    $ \mathbf{g}_b=\mathbf{e}+ H^4 + o(r^{-4}). $ Then, for $(\mathcal{L}_{X}\mathbf{g}_b)^\circ\in \mathbf{O}(\mathbf{g}_b)$ defined in Proposition \ref{def by scaling}, we have: $$(\mathcal{L}_{X}\mathbf{g}_b)^\circ = -2 H^4 + o(r^{-4})$$
\end{lem}
\begin{proof}
Let us denote $\Omega_s'$ the open interior of the hypersurface $\Psi(s\mathbb{S}^3\slash\Gamma)$ for the volume gauge diffeomorphism $\Psi:(\mathbb{R}^4\slash\Gamma)\backslash B_\mathbf{e}(s_0) \mapsto N$ of Definition \ref{def volume gauge} for $s_0$ large enough. All along the proof we will abusively omit the diffeomorphism $\Psi$ when pulling back tensors on the infinity of $N$ to $\mathbb{R}^4\slash\Gamma$. By Definition \ref{def volume gauge}, we have as $s\to \infty$
\begin{equation}
    \vol_\mathbf{e}(B_\mathbf{e}(s)\slash\Gamma)= \vol_{\mathbf{g}_b}\Omega'_s+o(1)\label{propriété volume gauge}
\end{equation}

The first step in order to find the asymptotic of $(\mathcal{L}_{X}\mathbf{g}_b)^\circ$ is to recall that $X = \frac{1}{2} \nabla_{\mathbf{g}_b}u$ where $u$ satisfies $ \Delta_{\mathbf{g}_b}u = 8 $ with $u = r^2 + \frac{b}{r^2}+o(r^{-2})$ for some $b\in \mathbb{R}$ similarly to \cite{bh}. We determine $b$ by integrating by parts $\Delta_{\mathbf{g}_b} u = 8$ as in \cite{bh}:
    \begin{align*}
        8\vol_\mathbf{e}(B_\mathbf{e}(s)\slash\Gamma)&= 8 \vol_{\mathbf{g}_b}\Omega'_s+o(1)\\
        &=\int_{\Omega_s'} \Delta_{\mathbf{g}_b} u dv_{\mathbf{g}_b}+o(1)\\ 
        &= 8\vol_\mathbf{e}(B_\mathbf{e}(s)\slash\Gamma) -2b|\mathbb{S}^3\slash\Gamma| -2rH^4(\partial_r,\partial_r) + o(1).
    \end{align*}
    which gives us $ b= -r^4H^4(\partial_r,\partial_r) $ again mimicking computations of \cite{bh} in our slightly different coordinates. 
    
    \begin{rem}
        This is consistent with $\mathbf{eh}$ for which $X = r\partial_r = \frac{1}{2} \nabla_\mathbf{eh} \sqrt{1+r^4}$ where $u = \sqrt{1+r^4}$ satisfies $\Delta_{\mathbf{eh}} u=8$, and $\sqrt{1+r^4} = r^2 +\frac{1}{2r^2}+o(r^{-2})$. We moreover have $\Hess_{\mathbf{eh}}\sqrt{1+r^4} = 2(\mathbf{eh}+(\mathcal{L}_{X}\mathbf{eh})^\circ)$ according to \cite[Proof of Proposition 2.1]{biq1}.
    \end{rem}
     As in the case of $\mathbf{eh}$ in \eqref{eguchihanson}, where $\nabla_{\mathbf{eh}}\sqrt{1+r^4} = 2r\partial_r$, developing $\nabla_{\mathbf{g}_b} = \nabla_\mathbf{e} + \nabla_\mathbf{e}^{(1)}(H^4) +...$, we find 
    $\nabla_{\mathbf{g}_b}u = 2r\partial_r + o(r^{-4}).$
    This implies that at infinity, we have the stated result, namely: $$(\mathcal{L}_X\mathbf{g}_b)^\circ = \mathcal{L}_{r\partial_r} H^4 + o(r^{-4})= -2H^4+ o(r^{-4}).$$ 
\end{proof}
This lets us state the following obstruction result.
\begin{thm}\label{obstruction desing taub's}
    Let $(M_o,\mathbf{g}_o)$ be an Einstein orbifold with a singularity $\mathbb{R}^4\slash\Gamma$ and let $(N,\mathbf{g}_b)$ be a Ricci-flat ALE manifold with integrable Ricci-flat ALE deformations asymptotic to $\mathbb{R}^4\slash\Gamma$. Let us denote $H^4$ the asymptotic term of $\mathbf{g}_b$ in \emph{volume gauge} as in Definition \ref{def volume gauge}:
    $ \mathbf{g}_b = \mathbf{e} + H^4 + \mathcal{O}(r^{-5}), $
    and let $H_2$ be the quadratic terms of $\mathbf{g}_o = \mathbf{e}+H_2+\mathcal{O}(r^3)$ in a coordinate system in which $B_\mathbf{e}(H_2)=0$. 
    
    Assume that there exists a sequence of Einstein metrics, $(M,\mathbf{g}_n)_n$, $d_{GH}$-converging to $(M_o,\mathbf{g}_o)$ and such that there exists a sequence $(t_n)_{n\in\mathbb{N}}$ with $t_n>0$ satisfying:
    $ \big(M,\frac{\mathbf{g}_n}{t_n}\big) \xrightarrow[GH]{}(N,\mathbf{g}_b).$ Then, we have the following obstructions:
    \begin{itemize}
        \item for any Killing vector field $Y_i^\pm$,
        \begin{equation}
            \int_{\mathbb{S}^{3}} \big(\mathring{\Ric}^{(2)}_\mathbf{e}(H_2,H^4)\big)(Y_i^\pm,\partial_r)dv_{\mathbb{S}^{3}} = 0,\label{premiere equation yi}
        \end{equation}
        \item for the conformal Killing vector field $r\partial_r$, we have
        \begin{equation}
          \int_{\mathbb{S}^{3}} \big(\mathring{\Ric}^{(2)}_\mathbf{e}(H_2,H^4)\big)(r\partial_r,\partial_r)dv_{\mathbb{S}^{3}} = 0.\label{premiere equation rdr}
        \end{equation}
    \end{itemize}
\end{thm}
\begin{rem}
    As we will see in Corollary \ref{obst sph hyp base} proven later, the obstruction \eqref{premiere equation rdr} is \emph{never} trivial if $(N,\mathbf{g}_b)$ is not flat. That is, if $(N,\mathbf{g}_b)$ is not flat, then there are some $H_2$ as above for which \eqref{premiere equation rdr} is not satisfied.
\end{rem}
\begin{rem}\label{rem kro gauge}
    By \cite{auvray}, the asymptotic terms of Kronheimer's instantons in volume gauge are generally of the form:
    $ H^4 = -\sum_{i,j} \langle\zeta_i,\zeta_j\rangle \frac{\theta_i^-\circ\omega_j^+}{2r^4} $
    for $\zeta_1,\zeta_2,\zeta_3\in \mathbb{R}^k$ for some $k$ arbitrarily large depending on the group at infinity ($k=1$ for the Eguchi-Hanson metric). For $k\geqslant 3$ we can construct $\zeta_1,\zeta_2,\zeta_3\in \mathbb{R}^k$ with
    $\langle\zeta_i,\zeta_j\rangle = \delta_{ij}$ and we check that the obstructions \eqref{premiere equation yi} vanishes for \emph{any} $H_2$.
\end{rem}
\begin{proof}
    Let us come back to the origin of the obstructions of \cite{ozu2,ozuthese}. They are the obstructions to solving:
        \begin{equation}
        \left\{\begin{aligned}
             \mathring{\Ric}_{\mathbf{g}_b}^{(1)}(h_2) &=0.\\
             h_2&=H_2 + H^4_2 + \mathcal{O}(r^{-3+\epsilon}),
        \end{aligned}\right.\label{equation h2}
    \end{equation}
    where $|H^4_2|_\mathbf{e}\sim r^{-2}$, see \cite{ozu3}.

    Let us consider the harmonic vector field $X = r\partial_r + \mathcal{O}(r^{-3})$ and the harmonic vector fields $Y_i' = Y_i+\mathcal{O}(r^{-3})$ defined in Section \ref{section deformation ricci plat ale}. Following the computation of \cite{biq1,ozu2}, noting that in volume gauge, we have $(\mathcal{L}_X\mathbf{g}_b)^\circ = -2 H^4+o(r^{-4})$ by Lemma \ref{lem volume gauge et asymptotiques o_i} and we moreover have:  
    \begin{equation}
    \begin{aligned}
        0&=\frac{1}{2}\lim_{R\to+\infty}\int_{\{r\leqslant R\}} \left\langle \mathring{\Ric}_{\mathbf{g}_b}^{(1)}(h_2)\;,\;(\mathcal{L}_X\mathbf{g}_b)^\circ \right\rangle_{\mathbf{g}_b}dv_{\mathbf{g}_b}\\
       &=2\int_{\mathbb{S}^3\slash\Gamma}3\left\langle H_2, H^4\right\rangle + H^4(B_\mathbf{e}H_2,\partial_r) dv_{\mathbb{S}^3\slash\Gamma}
    \end{aligned}\label{obst o1 biq1}
    \end{equation}
    which is manifestly linear in $H^4$. Similarly, for any Killing vector field $Y_i$ and its harmonic extension $Y_i'$, we find the obstructions: 
    \begin{equation}
        \int_{\mathbb{S}^3\slash\Gamma}3\left\langle H_2, \mathcal{L}_{Y_i}H^4\right\rangle+ \mathcal{L}_{Y_i}H^4(B_\mathbf{e}H_2,\partial_r) dv_{\mathbb{S}^3\slash\Gamma}=0.\label{obst o23 biq1}
    \end{equation}
    
    Now, the interpretation of \cite{biq1,biq2} of the obstruction tells us that for some $C=C(\Gamma)\neq 0$, for any $k,l\in\{1,2,3\}$
    \begin{equation}
        \int_{\mathbb{S}^3\slash\Gamma} \left\langle H_2, \frac{\theta_k^\mp\circ\omega_l^\pm}{r^4}\right\rangle  dv_{\mathbb{S}^3\slash\Gamma} = C \left\langle\mathbf{R}^{\pm,(1)}_\mathbf{e}(H_2)(\omega_k^\pm),\omega_l^\pm \right\rangle.\label{obst a la biq1}
    \end{equation}
    By linearity, this lets us compute the obstruction for any $H^4$ coming from a volume gauge.
    
    On the other hand, let us use the formalism of \cite{biq2} and compute the second variation $\mathring{\Ric}^{(2)}_\mathbf{e}(H_2,H^4)$ when $B_\mathbf{e}(H_2) = 0$. For the quadratic terms $H_2$ of the orbifold, we decompose 
    $$\mathbf{R}^{+,(1)}_\mathbf{e}(H_2) = \sum_{ij}R^+_{ij}\omega_i^+\otimes\omega_j^+,$$
    and we will consider another term of the form
    $H^4 = \frac{\theta_l^-\circ\omega_k^+}{r^4}$. Using the formula \eqref{Ric2 si scal cst} for the second variation of the traceless part of the Ricci curvature denoted $\mathring{\Ric}$ applied to the hyperkähler flat metric $\mathbf{e}$ in the direction $H^4+H_2$, we find 
    \begin{equation}
        \mathring{\Ric}_\mathbf{e}^{(2)}(H^4,H_2) = \sum_{j}R^+_{kj}\frac{\theta_l^-\circ\omega_j^+}{r^4}\label{Ric2 H4H2}.
    \end{equation}
    Now, this lets us compute the obstruction \eqref{premiere equation yi}. For any $i\in \{1,2,3\}$, define $(il)\in\{1,2,3\}$ by $\omega_i^+\circ\omega_{(il)}^+ = \pm \omega_l^+.$ We find the following value for \eqref{premiere equation yi} with Killing vector field $Y_{(il)}^+=\omega_{(il)}^+(r\partial_r)$:
    \begin{equation}
    \begin{aligned}
        0&=\sum_{j}R^+_{kj}\int_{\mathbb{S}^3} \langle\theta_l^-(\partial_r),\omega_j^+(\omega_{(il)}^+(r\partial_r))\rangle_\mathbf{e}dv_{\mathbb{S}^3}\\
        &= \pm R_{ki}^+\int_{\mathbb{S}^3} \langle Y_l^-,Y_l^-\rangle_\mathbf{e}dv_{\mathbb{S}^3}.
    \end{aligned}
    \end{equation}
    We therefore see that if \eqref{obst o23 biq1} with $Y_{(il)}^+$, is satisfied, then, one has \eqref{premiere equation yi} satisfied thanks to \eqref{obst a la biq1}. Similarly, from the obstruction \eqref{premiere equation rdr}, we find:
    \begin{equation}
    \begin{aligned}
        0&=\sum_{j}R^+_{kj}\int_{\mathbb{S}^3} \langle\theta_l^-(\partial_r),\omega_j^+(r\partial_r)\rangle_\mathbf{e}dv_{\mathbb{S}^3}\\
        &= \sum_{j}R^+_{kj}\int_{\mathbb{S}^3} \langle Y_l^+,Y_j^+\rangle_\mathbf{e}dv_{\mathbb{S}^3}\\
        &= R^+_{kl}\int_{\mathbb{S}^3} \langle Y_l^+,Y_l^+\rangle_\mathbf{e}dv_{\mathbb{S}^3},
    \end{aligned}
    \end{equation}
    and we consequently see that the condition \eqref{premiere equation rdr} is the same as \eqref{obst o1 biq1} thanks to \eqref{obst a la biq1}. 
\end{proof}

\begin{rem}
    It might seem like the proof of Theorem \ref{obstruction desing taub's} is completely disjoint from the obstruction of Theorem \ref{obst hk a scal2 nul} or those of Section \ref{closing up EH}. We illustrate below that they actually build on estimates of the exact same quantity but rely on different integrations by parts. Indeed, as noticed along the proof of Lemma \ref{lem volume gauge et asymptotiques o_i} a special feature of the volume gauge is that $ \nabla_{\mathbf{g}_b}u \sim r\partial_r$ at an order higher than expected. 
\end{rem}

Let us illustrate this with the Eguchi-Hanson metric $\mathbf{eh}$ which was already discussed in Section \ref{closing up EH}. Coming back to the integral quantity in \eqref{obst o1 biq1} and integrating by parts \emph{in the other direction} thanks to \eqref{ipp div free 2tensor}, we find: 
    \begin{align}
        0&=\lim_{R\to+\infty}\int_{\{r=R\}}\big(\mathring{\Ric}_{\mathbf{eh}}^{(1)}(h_2)\big)(r\partial_r,\partial_r)dv_{\mathbf{eh}|\{r=R\}}\nonumber \\
        &= \frac{1}{2}\lim_{R\to+\infty}\int_{\{r\leqslant R\}} \left\langle \big(\mathring{\Ric}_{\mathbf{eh}}^{(1)}(h_2)\big),\mathcal{L}_{r\partial_r}\mathbf{eh} \right\rangle_{\mathbf{eh}}dv_{\mathbf{eh}}\nonumber\\
        &= \frac{1}{2}\lim_{R\to+\infty}\int_{\{r\leqslant R\}} \left\langle \big(\mathring{\Ric}_{\mathbf{eh}}^{(1)}(h_2)\big),(\mathcal{L}_{r\partial_r}\mathbf{eh})^\circ \right\rangle_{\mathbf{eh}}dv_{\mathbf{eh}}\label{annulation ipp}
    \end{align}
    The left-hand side term of \eqref{annulation ipp} has the following limit as $R\to +\infty$:
    \begin{align*}
        0&=\int_{\mathbb{S}^3\slash\mathbb{Z}_2}\Big( \E^{(1)}_{\mathbf{e}}(H_2^4) + \E_{\mathbf{e}}^{(2)}(H^4,H_2)+\lambda H^4\Big)(r\partial_r,\partial_r) dv_{\mathbb{S}^3\slash\mathbb{Z}_2}\\
        &=\int_{\mathbb{S}^3\slash\mathbb{Z}_2}\Big( \E^{(1)}_{\mathbf{e}}(H_2^4)\Big)(r\partial_r,\partial_r)dv_{\mathbb{S}^3\slash\mathbb{Z}_2} +\int_{\mathbb{S}^3\slash\mathbb{Z}_2}\Big(\mathring{\Ric}_{\mathbf{e}}^{(2)}(H^4,H_2)\Big)(r\partial_r,\partial_r) dv_{\mathbb{S}^3\slash\mathbb{Z}_2}.
    \end{align*}
    which is exactly the quantity obtained in \eqref{dvp eh desing}. As in Section \ref{closing up EH}, we may use $H^4_2 = 0$.
    \\
    
We recover \emph{all} of the obstructions of \cite{biq1}.
\begin{cor}\label{obst EH}
    With the assumptions of Theorem \ref{obstruction desing taub's} with $(N,\mathbf{g}_b) = (T^*\mathbb{S}^2,\mathbf{eh})$ defined in \eqref{eguchihanson} with $H^4 = -\frac{\theta_1^-\circ \omega_1^+}{2 r_\mathbf{e}^4}$ and any $H_2$ satisfying $B_\mathbf{e}(H_2)=0$, the following propositions are equivalent:
    \begin{enumerate}
        \item $\mathbf{R}^{+,(1)}_\mathbf{e}(H_2)\omega_1^+ = 0$,
        \item
        $
            \int_{\mathbb{S}^3\slash\mathbb{Z}_2} \big(\mathring{\Ric}^{(2)}_\mathbf{e}(H^4,H_2)\big)(r\partial_r,\partial_r) dv_{\mathbb{S}^3\slash\mathbb{Z}_2} = 0,$
        and for any $i\in\{2,3\}$ for the Killing vector field $Y_i^+$,
        $ \int_{\mathbb{S}^3\slash\mathbb{Z}_2} \big(\mathring{\Ric}^{(2)}_\mathbf{e}(H^4,H_2)\big)(Y_i^+,\partial_r) dv_{\mathbb{S}^3\slash\mathbb{Z}_2} = 0, $
        \item there is a solution $h_2$ to the equation \eqref{equation h2},
        \item 
        $\int_{\mathbb{S}^3\slash\mathbb{Z}_2}\left\langle H_2, H^4\right\rangle dv_{\mathbb{S}^3\slash\mathbb{Z}_2} = 0,$ and $i\in\{2,3\}$:
        $\int_{\mathbb{S}^3\slash\mathbb{Z}_2}\left\langle H_2, \mathcal{L}_{Y_i^+}H^4\right\rangle dv_{\mathbb{S}^3\slash\mathbb{Z}_2} = 0.$
    \end{enumerate}
\end{cor}

\subsection{Spherical and hyperbolic orbifolds}

Using the above new interpretation of the obstruction to the desingularization of Einstein $4$-manifolds as well as the integration by parts \eqref{ipp div free 2tensor}, we prove one of the main conjectures of \cite{ozu2,ozuthese} and answer the long-standing question of whether or not Einstein orbifolds can always be $d_{GH}$-desingularized by smooth Einstein metrics. 

The starting point is that the obstruction \eqref{premiere equation rdr} never vanishes when the orbifold is spherical or hyperbolic.
\begin{cor}\label{obst sph hyp base}
    If the obstruction \eqref{premiere equation rdr} vanishes, then, the orbifold $(M_o,\mathbf{g}_o)$ is not spherical or hyperbolic.
\end{cor}
\begin{proof}
Let us consider $H^4 = \sum_{ij}h_{ij}^+ \frac{\theta_i^-\circ\omega_j^+}{r^4}+\sum_{kl}h_{kl}^- \frac{\theta_k^+\circ\omega_l^-}{r^4}$. From \eqref{Ric2 H4H2}, by linearity, we have the following expression:
    \begin{equation}
        \mathring{\Ric}_\mathbf{e}^{(2)}(H^4,H_2) = \sum_{ijm}h_{ij}^+R^+_{jm}\frac{\theta_i^-\circ\omega_m^+}{r^4} +\sum_{kln}h_{kl}^-R^-_{ln}\frac{\theta_k^+\circ\omega_n^-}{r^4},\label{Ric2 H4H2 bis}
    \end{equation}
    hence, the obstruction \eqref{premiere equation rdr} rewrites:
    \begin{equation}
        \sum_{ij}h_{ij}^+R^+_{ji} +\sum_{kl}h_{kl}^-R^-_{lk} = 0.\label{somme obst}
    \end{equation}
    Denoting $\Lambda = \sum_iR^+_{ii} = \sum_k R^-_{kk}$, $(W^\pm_{ij})_{ij} $ the traceless part of $(R^\pm_{ij})_{ij}$, and $ \mathring{h}_{ij}^\pm $ the traceless part of $(h_{ij}^\pm)_{i,j\in\{1,2,3\}}$ (note that the $\mathring{h}_{ij}^\pm$ are proportional to the asymptotic curvature of $\mathbf{g}_b$ by Proposition \ref{asympt ALE jauge volume}) and recall that $\sum_{i}h_{ii}^++\sum_kh_{kk}^- = c(\Gamma) \mathcal{V}(N,\mathbf{g}_b)<0$.
    We therefore write \eqref{somme obst} as
    \begin{equation}
        0=\Lambda c(\Gamma) \mathcal{V}(N,\mathbf{g}_b) + \sum_{ij}\mathring{h}_{ij}^+ W^+_{ji} + \sum_{kl}\mathring{h}_{kl}^- W^-_{lk}.\label{obst en terme de courbure et Weyl}
    \end{equation}
    In the case of a spherical or hyperbolic orbifold, we have $W^\pm=0$ and $\Lambda\neq 0$, therefore, since $ c(\Gamma) \mathcal{V}(N,\mathbf{g}_b)<0 $, the obstruction \eqref{obst en terme de courbure et Weyl} is not satisfied.
\end{proof}

\subsubsection{Cokernel of the linearization of $\E$ at a Ricci-flat ALE metric}
Let us consider a Ricci-flat ALE orbifold $ (N,\mathbf{g}_b)$, the operator
$$g\mapsto \mathbf{\Phi}_{\mathbf{g}_b}(g) := \E(g)+\delta_{\mathbf{g}_b}^*\delta_{\mathbf{g}_b} g, $$
and $\mathbf{O}(\mathbf{g}_b)$ the $L^2(\mathbf{g}_b)$-kernel of $\mathbf{\Phi}_{\mathbf{g}_b}^{(1)}$, the linearization of $\mathbf{\Phi}_{\mathbf{g}_b}$ at $\mathbf{g}_b$. Note that the elements of $\mathbf{O}(\mathbf{g}_b)$ are traceless and divergence-free, see \cite{ozu2} for instance. According to \cite{ozu2,ozuthese}, for any small enough $v\in \mathbf{O}(\mathbf{g}_b)$, there exists a unique metric $g_v$ which satisfies both $$
        \mathbf{\Phi}_{\mathbf{g}_b}(g_v) = \E(g_v) + \delta^*_{\mathbf{g}_b}\delta_{\mathbf{g}_b}g_v\in \mathbf{O}(\mathbf{g}_b)
\;\text{ and }\; g_v - (\mathbf{g}_b+v)\perp_{L^2(\mathbf{g}_b)} \mathbf{O}(\mathbf{g}_b). $$

We will use the following Lemma in order to prove Proposition \ref{prop orthogonal lie der}.

\begin{lem}\label{terme bord premiere derivees}\label{claim proof}
    Assume that for all $k\leqslant l-1$,  $\partial^k_{s^k|s=0}\mathbf{\Phi}_{\mathbf{g}_b}(g_{sv}) =0$. Then, we have 
    \begin{enumerate}
    \item for all $k\leqslant l-1$, we have:
    $ \partial_{s^k|s=0}^k\E(g_{sv}) = 0, $ and for all $k\leqslant l$, we have $\delta_{\mathbf{g}_b}\partial_{s^k|s=0}^kg_{sv} = 0$,
        \item  $\delta_{\mathbf{g}_b}\Big(\partial^l_{s^l|s=0}\mathbf{\Phi}_{\mathbf{g}_b}(g_{sv}) - \mathbf{\Phi}_{\mathbf{g}_b}^{(1)}(\partial^l_{s^l|s=0}g_{sv})\Big)= 0$, and
        \item for some multilinear function $ Q_g^{(l)}(h_1,...,h_l) = \sum_{m\geqslant 2}\sum_{j_1,...,j_m\geqslant 1}^l \nabla_gh_{j_1}*h_{j_2}*...*h_{j_m}, $
where $*$ denotes various contractions of the tensors, we have:
$$\int_\Omega  \left\langle \mathbf{g}_b,\partial_{s^l{|s= 0}}^l \mathbf{\Phi}_{\mathbf{g}_b}(g_{sv}) - \mathbf{\Phi}_{\mathbf{g}_b}^{(1)}(\partial_{s^l{|s= 0}}^lg_{sv}) \right\rangle_{\mathbf{g}_b}dv_{\mathbf{g}_b} = \int_{\partial\Omega} Q^{(l)}(\partial_{s{|s= 0}}g_{sv},...,\partial_{s^{l-1}{|s= 0}}^{l-1}g_{sv}). $$
    \end{enumerate}
\end{lem}
\begin{proof}

Let us show these three properties in order.
\begin{enumerate}
    \item Notice that for any $k$, $\partial_{s^k}^k \mathbf{\Phi}_{\mathbf{g}_b}(g_{sv}) = \partial_{s^k|s=0}^k\E(g_{sv}) + \delta_{\mathbf{g}_b}^*\delta_{\mathbf{g}_b}(\partial_{s^k|s=0}^kg_{sv})$. Therefore, since $ \delta_{\mathbf{g}_b}\mathbf{\Phi}(g_{sv})=0 $ because $\mathbf{\Phi}(g_{sv})\in \mathbf{O}(\mathbf{g}_b)$, then one has: 
    $$\delta_{\mathbf{g}_b}\Big(\partial_{s^k|s=0}^k\E(g_{sv}) + \delta_{\mathbf{g}_b}^*\delta_{\mathbf{g}_b}(\partial_{s^k|s=0}^kg_{sv})\Big) = 0.$$
    If we moreover assume $\partial_{s^j}^j\E(g_{sv}) = 0$ for all $j\leqslant k-1$, we find
    $\delta_{\mathbf{g}_b}\delta_{\mathbf{g}_b}^*\delta_{\mathbf{g}_b}(\partial_{s^k|s=0}^kg_{sv}) = 0$ since $0 = \partial^k_{s^k|s=0}(\delta_{g_{sv}}\E(g_{sv})) = \delta_{\mathbf{g}_b}\partial^k_{s^k|s=0}\E(g_{sv})$. Since $\delta_{\mathbf{g}_b}\delta_{\mathbf{g}_b}^*$ is invertible on vector fields (or $1$-forms) decaying at infinity by \cite{biq1,ozu2}, we find $ \delta_{\mathbf{g}_b}(\partial_{s^k|s=0}^kg_{sv}) =0 $, but that means that $\partial_{s^k|s=0}^k\E(g_{sv}) = \partial_{s^k|s=0}^k \mathbf{\Phi}_{\mathbf{g}_b}(g_{sv}) - \delta_{\mathbf{g}_b}^*\delta_{\mathbf{g}_b}(\partial_{s^k|s=0}^kg_{sv}) =0$. We can iterate this up to $k = l-1$ since by assumption $ \partial_{s^k|s=0}^k \mathbf{\Phi}_{\mathbf{g}_b}(g_{sv}) =0$ for all $k\leqslant l-1$. This lets us also find $\delta_{\mathbf{g}_b}(\partial_{s^l|s=0}^lg_{sv}) =0$.
    \item For the equality $\delta_{\mathbf{g}_b}\Big(\partial^l_{s^l|s=0}\mathbf{\Phi}_{\mathbf{g}_b}(g_{sv}) - \mathbf{\Phi}_{\mathbf{g}_b}^{(1)}(\partial^l_{s^l|s=0}g_{sv})\Big)= 0$, we first see that $$\partial^l_{s^l|s=0}\mathbf{\Phi}_{\mathbf{g}_b}(g_{sv}) - \mathbf{\Phi}_{\mathbf{g}_b}^{(1)}(\partial^l_{s^l|s=0}g_{sv}) = \partial^l_{s^l|s=0}\E(g_{sv}) - \E_{\mathbf{g}_b}^{(1)}(\partial^l_{s^l|s=0}g_{sv})$$ because $\mathbf{\Phi} = \E + \delta_{\mathbf{g}_b}^*\delta_{\mathbf{g}_b}$. We conclude by noticing as above that $$\delta_{\mathbf{g}_b} \partial^l_{s^l|s=0}\E(g_{sv}) = \partial^l_{s^l|s=0}\Big(\delta_{g_{sv}} \E(g_{sv})\Big) = 0$$ because $\partial^k_{s^k|s=0}\E(g_{sv}) = 0$ for all $k\leqslant l-1$ and $\delta_{\mathbf{g}_b}\E_{\mathbf{g}_b}^{(1)}(\partial^l_{s^l|s=0}g_{sv})=0$ by Proposition \ref{gauge properties Bianchi}. 
    \item The first remark is that
    \begin{align*}
        \partial_{s^l{|s= 0}}^l \mathbf{\Phi}_{\mathbf{g}_b}(g_{sv}) =\;& \Big(\partial_{s^l{|s= 0}}^l \mathbf{\Phi}_{\mathbf{g}_b}(g_{sv}) - \mathbf{\Phi}_{\mathbf{g}_b}^{(1)}(\partial_{s^l{|s= 0}}^lg_{sv})\Big)\\
        &+ \E_{\mathbf{g}_b}^{(1)}(\partial_{s^l{|s= 0}}^lg_{sv})\\
        &+ \delta_{\mathbf{g}_b}^*\delta_{\mathbf{g}_b}\partial_{s^l{|s= 0}}^lg_{sv},
    \end{align*}
    and we will treat separately each of these terms.
    
    We then show, using \eqref{ipp courbure scalaire(1) general}, that the integral of the trace of $\E_{\mathbf{g}_b}^{(1)}(\partial_{s^l{|s= 0}}^lg_{sv})$ is a boundary term:
\begin{equation}
    \begin{aligned}
         \frac{2}{2-d}\int_\Omega  &\left\langle \mathbf{g}_b, \E_{\mathbf{g}_b}^{(1)}(\partial_{s^l{|s= 0}}^lg_{sv}) \right\rangle_{\mathbf{g}_b}dv_{\mathbf{g}_b} \\
         &= - \int_{\partial \Omega} \big(\delta_{\mathbf{g}_b}\partial_{s^l{|s= 0}}^lg_{sv}+d\tr_{\mathbf{g}_b} \partial_{s^l{|s= 0}}^lg_{sv}\big)(n_{\partial\Omega}) dv_{{\mathbf{g}_b}_{|\partial_\Omega}}.
    \end{aligned}\label{terme bord ordre 1}
\end{equation}
Let us now turn to the remaining linear term $\delta^*_{\mathbf{g}_b}\delta_{\mathbf{g}_b}\partial_{s^l{|s= 0}}^lg_{sv}$ in the expression of $\mathbf{\Phi}_{\mathbf{g}_b}^{(1)}(\partial_{s^l{|s= 0}}^lg_{sv})$. By integration by parts, denoting $V = \delta_{\mathbf{g}_b}\partial_{s^l{|s= 0}}^lg_{sv}$, we get:
$$ \int_\Omega  \left\langle \mathbf{g}_b,\delta^*_{\mathbf{g}_b}V \right\rangle_{\mathbf{g}_b}dv_{\mathbf{g}_b} = -2\int_\Omega \delta_{\mathbf{g}_b} V dv_{\mathbf{g}_b} = 2 \int_{\partial\Omega} \langle V,n\rangle_{\mathbf{g}_b} dv_{\mathbf{g}_b|\partial\Omega}, $$
which is again a boundary term. 

Noticing that $$ \partial_{s^{l}{|s= 0}}^{l} \mathbf{\Phi}_{\mathbf{g}_b}(g_{sv}) - \mathbf{\Phi}_{\mathbf{g}_b}^{(1)}(\partial_{s^l{|s= 0}}^lg_{sv})= \partial_{s^{l}{|s= 0}}^{l}\E(g_{sv}) -\E_{\mathbf{g}_b}^{(1)}(\partial_{s^l{|s= 0}}^lg_{sv}), $$ we can therefore focus on showing that the integral of $\partial_{s^{l}{|s= 0}}^{l} \E_{\mathbf{g}_b}(g_{sv})$ against $\mathbf{g}_b$ is equal to a boundary term. Since the $l-1$ first derivatives vanish, we have:
\begin{equation}
    \begin{aligned}
         \frac{2}{2-d}\int_\Omega  \left\langle \mathbf{g}_b,\partial_{s^{l}{|s= 0}}^{l} \E(g_{sv}) \right\rangle_{\mathbf{g}_b}dv_{\mathbf{g}_b} &= \frac{d^l}{ds^l}_{|s= 0}\int_\Omega \R(g_{sv})dv_{g_{sv}} \\
         &= -\frac{d^{l-1}}{ds^{l-1}}_{|s= 0}\int_\Omega \left\langle \E(g_{sv}),\partial_sg_{sv} \right\rangle_{g_{sv}} dv_{g_{sv}} \\
         &\;\;\;\;- \frac{d^{l-1}}{ds^{l-1}}_{|s= 0}\int_{\partial \Omega} \big(\delta_{g_{sv}}(\partial_{s}g_{sv})+d\tr_{g_{sv}} (\partial_{s}g_{sv})\big)(n_{\partial\Omega}) dv_{{g_{sv}}_{|\partial_\Omega}},
    \end{aligned}
\end{equation}
where the first term vanishes because $\partial^k_{s^k{|s= 0}}\E(g_{sv}) = 0$ for all $k\leqslant l-1$ by the first point. Together with \eqref{terme bord ordre 1}, this proves the result.
\end{enumerate}
\end{proof}
The point of Lemma \ref{terme bord premiere derivees} is that if the different $\partial^k_{s^k|s= 0}g_{sv}$ decay faster than $r^{-2+\epsilon}$ at infinity for some small $\epsilon>0$ in dimension $4$, then, the boundary term
$$ \lim_{r\to \infty}\int_{r\mathbb{S}^3\slash\Gamma}Q^{(l)}(\partial_{s{|s= 0}}g_{sv},...,\partial_{s^{l-1}{|s= 0}}^{l-1}g_{sv}) $$
which will appear in the following proof will always vanish.

\begin{proof}[Proof of Proposition \ref{prop orthogonal lie der}]
Using the results of Appendix \ref{sec dep analytic} (very close to the proof of \cite{koi}), we know that for any $v\in \mathbf{O}(\mathbf{g}_b)$ small enough, the map $s\in(-1,1)\mapsto g_{sv}$ is real-analytic in the so-called $ C^{2,\alpha}_{\beta}(\mathbf{g}_b) $-topology defined in Appendix \ref{sec fct space} for $0<\beta<2$ close to $2$, say $\beta = 1.9$. We mainly use the following consequence. We have a $C^{2,\alpha}_{\beta}(\mathbf{g}_b)$-converging development $g_{sv} = \mathbf{g}_b + sv +\sum_{k\geqslant 2}s^kw_k$ around $s=0$ where for any $k$, there exists $C= C(k)>0$ with for $l\in\{0,1,2\}$:
\begin{equation}
    r^{1.9+l}|\nabla_{\mathbf{g}_b}^lw_k|_{\mathbf{g}_b}\leqslant C.\label{controle wk}
\end{equation}
This induces a $ C^\alpha_{\beta+2} $-converging development (see again Appendix \ref{appendix norm analyticity}):
\begin{align*}
    \mathbf{\Phi}_{\mathbf{g}_b}(g_{sv}) =&\; s^2\big(\mathbf{\Phi}_{\mathbf{g}_b}^{(1)}(w_2)+\mathbf{\Phi}_{\mathbf{g}_b}^{(2)}(v,v)\big)&\in \mathbf{O}(\mathbf{g}_b)\\
    &+s^3\big(\mathbf{\Phi}_{\mathbf{g}_b}^{(1)}(w_3)+\big(2\mathbf{\Phi}_{\mathbf{g}_b}^{(2)}(w_2,v)+\mathbf{\Phi}_{\mathbf{g}_b}^{(2)}(v,v,v)\big)&\in \mathbf{O}(\mathbf{g}_b)\\
    &+...&\in \mathbf{O}(\mathbf{g}_b)\\
    &+\frac{s^l}{l!}\partial^l_{s^l|s=0}\mathbf{\Phi}_{\mathbf{g}_b}(g_{sv})&\in \mathbf{O}(\mathbf{g}_b)\\
    &+...&\in \mathbf{O}(\mathbf{g}_b)\\
\end{align*}

Now, by the analysis of the Fredholm properties of $\mathbf{\Phi}_{\mathbf{g}_b}$ in \cite{ozu2}, we know that the $\mathbf{\Phi}_{\mathbf{g}_b}^{(1)}(w_l) = \frac{1}{l!}\mathbf{\Phi}_{\mathbf{g}_b}^{(1)}(\partial^l_{s^l|s=0}g_{sv})$ are $L^2(\mathbf{g}_b)$-orthogonal to $\mathbf{O}(\mathbf{g}_b)$, hence we need to study the $L^2(\mathbf{g}_b)$-projection of $\partial^l_{s^l|s=0}\mathbf{\Phi}_{\mathbf{g}_b}(g_{sv})-\mathbf{\Phi}_{\mathbf{g}_b}^{(1)}(\partial^l_{s^l|s=0}g_{sv})$ on $\mathbf{O}(\mathbf{g}_b)$.

If $\mathbf{\Phi}_{\mathbf{g}_b}(g_{sv})$ is identically vanishing, then we are done. If not, assume that $\partial^k_{s^k|s=0}\mathbf{\Phi}_{\mathbf{g}_b}(g_{sv}) = 0$ for all $k\leqslant l-1$, but $\partial^l_{s^l|s=0}\mathbf{\Phi}_{\mathbf{g}_b}(g_{sv}) =: \mathbf{o}_v\neq 0$. We want to show that $\mathbf{o}_v\perp (\mathcal{L}_X\mathbf{g}_b)^\circ$ and $\mathbf{o}_v\perp \mathcal{L}_{Y'}\mathbf{g}_b$ for $Y'$ harmonic vector field asymptotic to a Killing vector field. 

Let us now apply our integration by parts formula \eqref{ipp div free 2tensor} to the divergence-free $2$-tensor $S^{(l)}(v):=\partial^l_{s^l|s=0}\mathbf{\Phi}_{\mathbf{g}_b}(g_{sv}) - \mathbf{\Phi}_{\mathbf{g}_b}^{(1)}(\partial^l_{s^l|s=0}g_{sv})$, this gives
\begin{align*}
    0=&\;\lim_{r\to\infty}\int_{r\mathbb{S}^3/\Gamma} S^{(l)}(v)(X,\partial_r)dv_{r\mathbb{S}^3/\Gamma}\\
    =&\; \frac{1}{2}\int_N \left\langle S^{(l)}(v),(\mathcal{L}_{X}\mathbf{g}_b)^\circ \right\rangle_{\mathbf{g}_b}dv_{\mathbf{g}_b} \\
&+\int_N  \tr_{\mathbf{g}_b}S^{(l)}(v)dv_{\mathbf{g}_b},
\end{align*}
where the first boundary term vanishes because it is a finite linear combination of terms of the form $\E^{(k)}_{\mathbf{g}_b}(\partial_{s^{j_1}}^{j_1}g_{sv},...,\partial_{s^{j_k}}^{j_k}g_{sv})$ with $k\geqslant 2$ and $j_i\geqslant 1$. Indeed, we know from the results of Section \ref{sec dep analytic} in the appendix (or \eqref{controle wk}) that $\partial_{s^{j}}^{j}g_{sv}\in C^{2,\alpha}_{1.9}$ for every $j\geqslant 1$, which implies that for any $a\in \{0,1,2\}$, $r^{a+1.9}|\nabla^a\partial_{s^{j}}^{j}g_{sv}|\leqslant C_j$ for some $C_j>0$. This gives: $\E^{(k)}_{\mathbf{g}_b}(\partial_{s^{j_1}}^{j_1}g_{sv},...,\partial_{s^{j_k}}^{j_k}g_{sv}) = \mathcal{O}(r^{-2-k\cdot 1.9}) = o(r^{-4})$.
Similarly, the last term $\int_N  \tr_{\mathbf{g}_b}S^{(l)}(v)dv_{\mathbf{g}_b}$ is a boundary term of the same type by Lemma \ref{terme bord premiere derivees} and it vanishes for the same reason.

The proof for the different $\mathcal{L}_{Y'}\mathbf{g}_b$ is similar and easier because there is no trace term to deal with.
\end{proof}
\begin{rem}
    Again, the proof is very close to that of Theorem \ref{obst hk a scal2 nul} because in volume gauge, one has $ \nabla_{\mathbf{g}_b}u \sim r\partial_r$ at an order higher than expected.
\end{rem}

\subsubsection{Obstruction to the desingularization of spherical and hyperbolic orbifolds.}
Let us now prove our main result.
\begin{proof}[Proof of Theorem \ref{obst sph hyp}]
Let us assume that there exists a sequence of Einstein metrics $(M,\mathbf{g}_n)_n$ converging to an Einstein orbifold $(M_o,\mathbf{g}_o)$ with $\E(\mathbf{g}_o)+\lambda\mathbf{g}_o=0$. Then, according to \cite{ozu1,ozu2,ozu3}, up to taking a subsequence, there exist $(t_n)_n$, $t_n>0$, $(v_n)_n$, $v_n\in \mathbf{O}(\mathbf{g}_b)$ such that $(M,\mathbf{g}_n/t_n)_n$ is close to $(N,\mathbf{g}_b) $ in the following sense: for all $r\ll t_n^{-\frac{1}{2}}$, we have for some $0<\beta<1$
\begin{equation}
    (1+r)^{k+\beta}\Big|\nabla_{\mathbf{g}_b}^k\Big(\frac{\mathbf{g}_n}{t_n} - \big(g_{v_n} + t_n h_2\big)\Big)\Big|_{\mathbf{g}_b} = o(t_n)\label{controle dégén einstein}
\end{equation}
where $h_2$ is a solution to the following equation:
\begin{equation}
        \mathbf{\Phi}_{\mathbf{g}_b}^{(1)}(h_2)+\lambda \mathbf{g}_b \in  \mathbf{O}(\mathbf{g}_b).
\end{equation}

It was shown in \cite{ozu2} that the first obstruction against $(\mathcal{L}_{X}\mathbf{g}_b)^\circ$ to the desingularization is:
\begin{equation}
    \int_N\left\langle \mathbf{\Phi}_{\mathbf{g}_b}(\mathbf{g}_n/t_n)+t_n\lambda (\mathbf{g}_n/t_n) , \chi_{t_n}(\mathcal{L}_{X}\mathbf{g}_b)^\circ \right\rangle_{\mathbf{g}_b}dv_{\mathbf{g}_b}=0.\label{obst ozu2 seconde ecriture}
\end{equation}
Or goal is to show that it cannot be satisfied when desingularizing a spherical or hyperbolic orbifold.
\\

Recall from Proposition \ref{prop orthogonal lie der} that one has $\mathbf{\Phi}(g_{v_n}) = \mathbf{o}_{v_n} + \mathcal{O}(\|v_n\|_{L^2(\mathbf{g}_b)}\|\mathbf{o}_{v_n}\|_{L^2(\mathbf{g}_b)})\in \mathbf{O}(\mathbf{g}_b)$. If $ \mathbf{o}_{v_n} = 0$, then we are done by \cite{ozuthese}, let us assume that it is not satisfied. The obstruction against the deformation $\mathbf{w}_{v_n}:=\frac{\mathbf{o}_{v_n}}{\|\mathbf{o}_{v_n}\|_{L^2(\mathbf{g}_b)}}$ is \begin{equation}
    \int_N\left\langle \mathbf{\Phi}_{\mathbf{g}_b}(\mathbf{g}_n/t_n)+t_n\lambda (\mathbf{g}_n/t_n) , \chi_{t_n}\mathbf{w}_{v_n} \right\rangle_{\mathbf{g}_b}dv_{\mathbf{g}_b}=0.\label{obst ozu2 seconde ecriture vn}
\end{equation}

Let us now estimate \eqref{obst ozu2 seconde ecriture}. We use the control \eqref{controle dégén einstein} which tells us that:
\begin{equation}
    \mathbf{\Phi}_{\mathbf{g}_b}(\mathbf{g}_n/t_n)+t_n\lambda(\mathbf{g}_n/t_n) = \mathbf{\Phi}_{\mathbf{g}_b}(g_{v_n}) + t_n\mathbf{\Phi}^{(1)}_{\mathbf{g}_b}(h_2) + t_n\lambda \mathbf{g}_b +o(t_n (1+r)^{-2-\beta}).\label{dvp Phi h_n}
\end{equation}
We will therefore decompose \eqref{obst ozu2 seconde ecriture} and estimate each part of the integral.
\begin{enumerate}
    \item For the first integral
    $\int_N\left\langle\mathbf{\Phi}(g_{v_n}) , \chi_{t_n} (\mathcal{L}_{X}\mathbf{g}_b)^\circ \right\rangle_{\mathbf{g}_b}dv_{\mathbf{g}_b}$
    we use the estimate $$\int_N\left\langle\mathbf{\Phi}(g_{v_n}) , (\mathcal{L}_{X}\mathbf{g}_b)^\circ \right\rangle_{\mathbf{g}_b}dv_{\mathbf{g}_b}=\mathcal{O}(\|v_n\|_{L^2}\|\mathbf{o}_{v_n}\|_{L^2})$$ together with $\mathbf{\Phi}(g_{v_n}) = \mathcal{O}(\|\mathbf{o}_{v_n}\|_{L^2}r^{-4})$, which gives
    \begin{align*}
        \int_N\left\langle\mathbf{\Phi}(g_{v_n}) , (1-\chi_{t_n}) (\mathcal{L}_{X}\mathbf{g}_b)^\circ \right\rangle_{\mathbf{g}_b}dv_{\mathbf{g}_b} &= \mathcal{O}\Big(\int_{t_n^{-\frac{1}{4}}}^\infty\|\mathbf{o}_{v_n}\|_{L^2}r^{-6} r^{-4} r^3dr \Big)
        \\
        &=  \mathcal{O}(t_n\|\mathbf{o}_{v_n}\|_{L^2}).
    \end{align*}
    We finally find
    \begin{equation}
        \int_N\left\langle\mathbf{\Phi}(g_{v_n}) , \chi_{t_n} (\mathcal{L}_{X}\mathbf{g}_b)^\circ \right\rangle_{\mathbf{g}_b}dv_{\mathbf{g}_b} = \mathcal{O}((\|v_n\|_{L^2}+t_n)\|\mathbf{o}_{v_n}\|_{L^2}),
    \end{equation}
    \item let us denote $\lambda_1:=\int_N\left\langle\mathbf{\Phi}^{(1)}_{\mathbf{g}_b}(h_2)+ \lambda \mathbf{g}_b, (\mathcal{L}_{X}\mathbf{g}_b)^\circ \right\rangle_{\mathbf{g}_b}dv_{\mathbf{g}_b}$. We find:
    $$t_n\int_N \left\langle\mathbf{\Phi}^{(1)}_{\mathbf{g}_b}(h_2) + \lambda \mathbf{g}_b, \chi_{t_n}(\mathcal{L}_{X}\mathbf{g}_b)^\circ \right\rangle_{\mathbf{g}_b}dv_{\mathbf{g}_b} = t_n\lambda_1 + o(t_n), $$
    \item for the error term, we have $o(t_n)\int_N \left\langle (1+r)^{-2-\beta}, \chi_{t_n}(\mathcal{L}_{X}\mathbf{g}_b)^\circ \right\rangle_{\mathbf{g}_b}dv_{\mathbf{g}_b} = o(t_n)$.
\end{enumerate}

From \eqref{obst ozu2 seconde ecriture}, we finally find the estimate: 
\begin{equation}
    0 = t_n \lambda_1+\mathcal{O}(\|v_n\|_{L^2}\|\mathbf{o}_{v_n}\|_{L^2})+o(t_n).\label{estimate o1}
\end{equation}

Similarly, using again \eqref{dvp Phi h_n}, we then estimate \eqref{obst ozu2 seconde ecriture vn} in three parts:
\begin{enumerate}
    \item For the first integral
    $\int_N\left\langle\mathbf{\Phi}(g_{v_n}) , \chi_{t_n}\mathbf{w}_{v_n} \right\rangle_{\mathbf{g}_b}dv_{\mathbf{g}_b}$, we use Proposition \ref{prop orthogonal lie der}
    which implies that $$\int_N\left\langle\mathbf{\Phi}(g_{v_n}) , \mathbf{w}_{v_n} \right\rangle_{\mathbf{g}_b}dv_{\mathbf{g}_b} = \|\mathbf{o}_{v_n}\|_{L^2(\mathbf{g}_b)}+\mathcal{O}(\|v_n\|_{L^2(\mathbf{g}_b)}\|\mathbf{o}_{v_n}\|_{L^2(\mathbf{g}_b)}).$$ Since $\mathbf{\Phi}(g_{v_n}) = \mathcal{O}(\|\mathbf{o}_{v_n}\|_{L^2}r^{-4})$, we estimate the difference
    \begin{align*}
        \int_N\left\langle\mathbf{\Phi}(g_{v_n}) , (1-\chi_{t_n})\mathbf{w}_{v_n} \right\rangle_{\mathbf{g}_b}dv_{\mathbf{g}_b} &= \mathcal{O}\Big(\int_{t_n^{-\frac{1}{4}}}^\infty\|\mathbf{o}_{v_n}\|_{L^2}r^{-6} r^{-4} r^3dr \Big)
        \\
        &= \mathcal{O}(t_n\|\mathbf{o}_{v_n}\|_{L^2}),
    \end{align*}
    and finally
    \begin{equation}
        \int_N\left\langle\mathbf{\Phi}(g_{v_n}) , \chi_{t_n}\mathbf{w}_{v_n} \right\rangle_{\mathbf{g}_b}dv_{\mathbf{g}_b} = \|\mathbf{o}_{v_n}\|_{L^2} + \mathcal{O}((t_n+\|{v_n}\|_{L^2})\|\mathbf{o}_{v_n}\|_{L^2}),
    \end{equation}
    \item let us denote $\mu_{v_n}:=\int_N\left\langle\mathbf{\Phi}^{(1)}_{\mathbf{g}_b}(h_2)+ \lambda \mathbf{g}_b,\mathbf{w}_{v_n} \right\rangle_{\mathbf{g}_b}dv_{\mathbf{g}_b}$. We find:
    $$t_n\int_N\left\langle\mathbf{\Phi}^{(1)}_{\mathbf{g}_b}(h_2) + \lambda \mathbf{g}_b, \chi_{t_n}\mathbf{w}_{v_n}\right\rangle_{\mathbf{g}_b}dv_{\mathbf{g}_b} = t_n\mu_{v_n} + o(t_n), $$
    \item for the error term, we have $o(t_n)\int_N \left\langle (1+r)^{-2-\beta}, \chi_{t_n}\mathbf{w}_{v_n}\right\rangle_{\mathbf{g}_b}dv_{\mathbf{g}_b} = o(t_n)$.
\end{enumerate}
From \eqref{obst ozu2 seconde ecriture vn}, we finally find the estimate 
\begin{equation}
    0 = \|\mathbf{o}_{v_n}\|_{L^2(\mathbf{g}_b)} + t_n \mu_{v_n}+o(t_n+\|\mathbf{o}_{v_n}\|_{L^2(\mathbf{g}_b)}).\label{estimate vn}
\end{equation}
which gives $\|\mathbf{o}_{v_n}\|_{L^2(\mathbf{g}_b)} =\mathcal{O}(t_n)$ because $\mu_{v_n}$ is bounded since $\mathbf{O}(\mathbf{g}_b)$ is finite-dimensional. 

Finally by plugging $\|\mathbf{o}_{v_n}\|_{L^2(\mathbf{g}_b)} =\mathcal{O}(t_n)$ in \eqref{estimate o1}, we find 
$ |\lambda_1| =\mathcal{O}(\|v_n\|_{L^2}) $.
Now, since $\lambda_1$ is a constant, and since $(t_n,v_n)\to 0$, we obtain $\lambda_1 = 0$. It is impossible to satisfy if $H_2$ is the quadratic term of the development of a spherical or hyperbolic orbifold by Corollary \ref{obst sph hyp base}.
\end{proof}
\begin{rem}
    The above proof applies indifferently to other Einstein orbifolds and to the other deformations $\mathcal{L}_{Y'}\mathbf{g}_b$. In particular, it shows that the integrability assumption in Theorem \ref{obstruction desing taub's} is superfluous.
\end{rem}

We believe that the result should hold for spherical and hyperbolic orbifolds with more general singularities than $\mathbb{R}^4\slash\mathbb{Z}_2$, but this requires dealing with trees of singularities. The main difficulty is that it is not known whether the projection on the obstruction is a real-analytic map or not in this degenerate situation.
\begin{conj}
    Let $(M_o,\mathbf{g}_o)$ be a singular spherical or hyperbolic compact orbifold. Then, it is not limit of smooth Einstein metrics in the Gromov-Hausdorff sense.
\end{conj}

\subsection{Higher dimensional Einstein orbifolds with isolated singularities}

The work of \cite{ozu1,ozu2} extends almost verbatim to the degeneration of Einstein $d$-manifolds ($d\geqslant 5$) satisfying the (non natural) assumption that the $L^{\frac{d}{2}}$-norm of its curvature is bounded. It shows that, exactly like in dimension $4$, the possible Gromov-Hausdorff limits are Einstein orbifolds with isolated singularities and the singularity models are Ricci-flat ALE orbifolds. Indeed the results of \cite{ozu1,ozu2,ozuthese} only use the fact that the dimension is $4$ to obtain a bound on the $L^2$-norm of the Riemannian curvature from the noncollapsedness assumption.

An obstruction to the desingularization under essentially the same assumptions as \cite{biq1} was proven in \cite{mv} for higher dimensional desingularizations. Namely, one considers the desingularization by the so-called \emph{Calabi metric} denoted $\mathbf{g}_{cal}$ which is $2d$-dimensional, Ricci-flat ALE and asymptotic to $\mathbb{R}^{2d}\slash\mathbb{Z}_d$.

\begin{lem}
    The kernel $\mathbf{O}(\mathbf{g}_{cal})$ is $1$-dimensional and spanned by $(\mathcal{L}_X\mathbf{g}_{cal})^\circ$, where $X$ is a harmonic vector field asymptotic to $r\partial_r$.
\end{lem}
\begin{proof}
    The proof that $\mathbf{O}(\mathbf{g}_{cal})$ is $1$-dimensional is found in \cite{mv}. We therefore simply have to prove the existence of $X$ and the fact that $(\mathcal{L}_X\mathbf{g}_{cal})^\circ$ is divergence-free.
    
    Like in \cite{ozuthese,bh}, we consider the unique function $u=r^2+\mathcal{O}(r^{-2d+2})$ satisfying $\Delta_{\mathbf{g}_{cal}} u = 2d$. One then defines $X = \frac{1}{2}\nabla_{\mathbf{g}_{cal}} u = r\partial_r+\mathcal{O}(r^{-2d+1})$ which satisfies 
    $$\delta_{\mathbf{g}_{cal}}(\mathcal{L}_X\mathbf{g}_{cal})^\circ = 0$$
    by construction.
    
    A last step is to ensure that $(\mathcal{L}_X\mathbf{g}_{cal})^\circ\neq 0$ following \cite{bh}. If $(\mathcal{L}_X\mathbf{g}_{cal})^\circ = 0$, then since $\Delta_{\mathbf{g}_{cal}} u = 2d$, $X$ would generate $1$-parameter group of homotheties. By considering the maximum of the curvature tensor, this is impossible for the non flat metric $\mathbf{g}_{cal}$.
\end{proof}

This implies that the obstruction found in \cite{mv} is of the same type.
\begin{cor}\label{mv obst}
    Let $(M_o,\mathbf{g}_o)$ be an Einstein orbifold with a singularity $\mathbb{R}^{2d}\slash\mathbb{Z}_d$ at $p_o$. The obstruction to the desingularization of $(M_o,\mathbf{g}_o)$ at $p_o$ by $\mathbf{g}_{cal}$ found in \cite{mv}, namely
    $$ d\cdot\left\langle\mathbf{R}_{\mathbf{g}_o}(p_o)\omega,\omega \right\rangle + 2(d-2) \R_{\mathbf{g}_o}(p_o) = 0$$
    is equivalent to the obstruction against $(\mathcal{L}_{X}\mathbf{g}_b)^\circ$.
\end{cor}

\appendix

\section{Development of Einstein $4$-metrics}\label{appendix mf}

Let $(M_o,\mathbf{g}_o)$ be an Einstein orbifold (smooth or singular) and assume that at a point $p$, it has a development:
$\mathbf{g}_o=\mathbf{e}+H_2+\mathcal{O}(r^3).$
We start by showing that up to a \emph{gauge} term, the term $H_2$ has a good correspondence with the curvature.
\subsection{A local gauge for Einstein metrics}
\begin{prop}\label{terme quadratique orbifold}
    Let $H_2$ be a quadratic symmetric $2$-tensor satisfying 
    $\Ric^{(1)}_\mathbf{e}(H_2) = \Lambda\mathbf{e},$
    for $\Lambda\in \mathbb{R}$ and such that:
    $$ \mathbf{R}^{\pm,(1)}_\mathbf{e}(H_2) = \frac{\Lambda}{3}\sum_{i}\omega_i^\pm\otimes\omega_i^\pm + \sum_{ij}W_{ij}^\pm\omega_i^\pm\otimes\omega_j^\pm $$
    where we identified $\Omega^+_\mathbf{e}\otimes\Omega^+_\mathbf{e} \sim (\Omega^+_\mathbf{e})^*\otimes\Omega^+_\mathbf{e} \sim \operatorname{End}(\Omega^+_\mathbf{e})$ and where the $W^\pm_{ij}$ are the coefficients of the (anti-)selfdual Weyl curvature.
    
    Then, there exists a cubic vector field $V_3$ such that 
    \begin{equation}
        H_2 =-\frac{\Lambda}{9}r^4g_{\mathbb{S}^3} +\frac{r^2}{6}\Big(\sum_{ij}W_{ij}^+\theta_i^{-}\circ\omega_j^+ + \sum_{kl}W_{kl}^-\theta_k^{+}\circ\omega_l^-\Big) + \mathcal{L}_{V_3}\mathbf{e}.
    \end{equation}
    where $g_{\mathbb{S}^3}$ is the usual round metric on the unit $3$-sphere.
\end{prop}
\begin{proof}
    Let us first show that the infinitesimal curvature induced by $$\tilde{H}_2:=-\frac{\Lambda}{9}r^4g_{\mathbb{S}^3}+\frac{r^2}{6}\Big(\sum_{ij}W_{ij}^+\theta_i^{-}\circ\omega_j^+ + \sum_{kl}W_{kl}^-\theta_k^{+}\circ\omega_l^-\Big)$$ is the same as that of $H_2$. For the spherical metric in geodesic coordinates, $ g_{\mathbb{S}^4} $ with $\Lambda =3$, one has the development:
    \begin{equation}
        g_{\mathbb{S}^4} = \mathbf{e} - \frac{1}{3}r^4g_{\mathbb{S}^3} + \mathcal{O}(r^{3}),\label{dvp gS4 1}
    \end{equation}
    and therefore, by linearity of $H_2\mapsto \mathbf{R}^{\pm,(1)}_\mathbf{e}(H_2)$, we just have to deal with Ricci-flat deformations and their Weyl curvature. We rely on the formalism of \cite{biq2} for this computation. 
    
    We first note that each term $r^2\theta_i^\mp\circ\omega_j^\pm$ is traceless and harmonic, hence is in the kernel of $\mathring{\Ric}_\mathbf{e}^{(1)}$. We have the following formula
    $d(r^2\theta_i^-) = 6r dr\wedge \theta_i^-,$
    hence 
    \begin{equation}
        *d(r^2\theta_i^-) = -6r^2\alpha_i^+.\label{d r2theta}
    \end{equation}
    From this, we see that the term is in the Bianchi gauge \eqref{biachi gauge} since $W^+_{ij}=W^+_{ji}$, and
    $ -\frac{r^2}{6}\sum_{i}W^+_{ii}(-6\omega_i^+(\alpha_i^+)) = -r\sum_iW_{ii}^+dr = 0 $
    because $r\omega_i^+(\alpha_i^+) = -dr$ and $ \sum_iW_{ii}^+=0$. 
    
    From \eqref{d r2theta}, we moreover obtain:
    \begin{equation}
        -d*d(r^2\theta_i^-) = 6\omega_i^+,\label{d*d r2 thetai+}
    \end{equation}
    and therefore we have the following curvature induced by $H_2^+:=\frac{r^2}{6}\sum_{ij}W_{ij}^+\theta_i^{-}\circ\omega_j^+$:
    \begin{itemize}
        \item $\mathring{\Ric}^{(1)}_\mathbf{e}(H_2^+) = 0$,
        \item $\mathbf{R}^{+,(1)}_\mathbf{e}(H_2^+) = \begin{bmatrix}
       W^+_{11}&W^+_{12}&W^+_{13}\\
       W^+_{21}&W^+_{22}&W^+_{23}\\
       W^+_{31}&W^+_{32}&W^+_{33}
    \end{bmatrix}$.
    \end{itemize}

    Let us now prove that the induced anti-selfdual curvature $\mathbf{R}^{-,(1)}_\mathbf{e}(H_2^+)$ vanishes. On the other hand that seeing, $\mathbf{e}$ as a hyperkähler metric with the opposite orientation, we have the expression
    $r^2\theta_i^- = \sum_l \left\langle\omega_i^+(r\partial_r),\omega_l^-(r\partial_r)\right\rangle\omega_l^-, $
    hence $$r^2\theta_i^-\circ \omega_j^+ = \sum_l \left\langle\omega_i^+(r\partial_r),\omega_l^-(r\partial_r)\right\rangle\omega_l^-\circ \omega_j^+,$$ and we also find
    \begin{equation}
        d\Big(\sum_l \left\langle\omega_i^+(r\partial_r),\omega_l^-(r\partial_r)\right\rangle\omega_l^-\Big) = -\sum_l\big(\omega_l^-\circ\omega_i^+(rd_r)\big)\wedge\omega_j^+.\label{exp d(...)}
    \end{equation}
    Now, recall that for any $1$-form $\beta$, $ * (\beta\wedge\omega_j^+) = \omega_j^+(\beta)$, where we identify $\omega_j^+$ and the associated endomorphism by the metric. By \eqref{exp d(...)}, this gives
    $$*d\Big(\sum_l \left\langle\omega_i^+(r\partial_r),\omega_l^-(r\partial_r)\right\rangle\omega_j^+\Big) = -\sum_l\omega_j^+\circ\omega_i^+\circ\omega_l^-(rd_r). $$
    From the expression of
    $\mathbf{R}^{-,(1)}_\mathbf{e}(H_2^+) (\omega_l^-)$ and using that $W^+_{ij} = W^+_{ji}$, $\omega_j^+\circ\omega_i^+ =-\omega_i^+\circ\omega_j^+$ as well as $\sum_i W^+_{ii}= 0$, we find: $\mathbf{R}^{-,(1)}_\mathbf{e}(H_2^+) = 0$. The proof is exactly the same for the rest of the tensor coming from $W^-$.
    \\
    
    Let us now show that a quadratic $2$-tensor satisfying $\mathbf{R}_\mathbf{e}^{(1)}(H'_2)=0$ is necessarily of gauge type, that is: $H'_2 = \mathcal{L}_{V_3}\mathbf{e}$ for some cubic vector field $V_3$. According to \cite[(28)]{biq1}, there exists $V_3$ such that $H''_2:=H'_2 - \mathcal{L}_{V_3}\mathbf{e}$ is \emph{radial}, that is $H''_2(\partial_r,.)=0$, and in particular, there exist $H_{ij}$ with 
    \begin{equation}
        H''_2 = r^4\sum_{ij}H_{ij}\alpha_i^+\alpha_j^+\label{dvp radial}
    \end{equation}
    and still $\mathbf{R}_\mathbf{e}^{(1)}(H''_2)=0$. We now need to prove that $H''_2=0$.
    
    Now according to \cite[(38)]{biq1}, from \eqref{dvp radial}, one has:
    \begin{equation}
        0=\mathbf{R}_\mathbf{e}^{+,(1)}(H''_2)= -6\sum_{ij}H_{ij}\omega_{i}^+\otimes\omega_{j}^++(H_{11}+H_{22}+H_{33}) \textup{Id}_{\Omega_\mathbf{e}^+}.\label{courbure biquard}
    \end{equation}
    We therefore directly find $H_{ij}=0$ when $i\neq j$, and taking the trace of \eqref{courbure biquard}, we get
    $ 0=-6(H_{11}+H_{22}+H_{33}) + 3(H_{11}+H_{22}+H_{33}) $ and consequently $H_{11}+H_{22}+H_{33} = 0$. Finally, we see that $ H''_2 =0 $ and this ends the proof.
\end{proof}

\subsection{Vanishing of the obstructions to Einstein deformations}
Let us first deal with the case of Einstein deformations.

\begin{prop}
    For any quadratic $2$-tensor $H_2$ on $(\mathbb{R}^4,\mathbf{e})$ with: $\Ric_\mathbf{e}(H_2) = \Lambda \mathbf{e}$ for $\Lambda \in \mathbb{R}$, the obstructions  \eqref{obst taub} and \eqref{annulation T cas euclidien} vanish.
\end{prop}
\begin{rem}
    The result is true from Propositions \ref{integrability taub} and \ref{seconde obst conforme} and by \cite{gas}. We however prefer to give another much simpler way to see that it holds.
\end{rem}
\begin{proof}[Sketch of proof]
    We only sketch the proof as the result can essentially be found by bilinearity of the obstructions \eqref{annulation T cas euclidien} and \eqref{obst taub einstein} and thanks to the curvature of the known examples of Einstein metrics.
    
    Let us use the bilinear nature of our obstructions and decompose any quadratic $2$-tensor as in Proposition \ref{terme quadratique orbifold}, that is as $H_2 = H_{\R}+ H_{W_+}+ H_{W_-} + H_{0}$ where $ H_{\R} =\frac{\Lambda r^4}{9}g_{\mathbb{S}^3} $, $H_{W_\pm} = \frac{r^2}{6}\sum W^\pm_{ij}\theta_i^\mp\circ \omega_j^\pm$ and $H_{0}=\mathcal{L}_V\mathbf{e}$ for some vector field $V$ satisfying $|V|_\mathbf{e}\sim r^3$.
    
    Note that any term combined with a gauge term $H_0$ will make the obstructions vanish by invariance, see Proposition \ref{prop B T} and Remark \ref{invariance qté}. There remain several situations which can be settled thanks to the known examples of Einstein metrics:
    \begin{itemize}
        \item the obstructions for $(H_{\R},H_{\R})$ vanish because they do on the sphere,
        \item the obstructions of the form $(H_{\R}, H_{W_\pm})$ vanish by bilinearity and because of the examples of  (anti-)selfdual Kähler-Einstein metrics such as the Fubini-Study metric on $\mathbb{C}P^2$,
        \item the obstructions of the form $(H_{W_\pm}, H_{W_\pm})$ vanish by bilinearity and because of the examples of orbifold hyperkähler metrics such as the ones produced in \cite{ozu3}, where it is clear that the (anti-)selfdual curvature can take arbitrary values,
        \item lastly, the obstructions of the form $(H_{W_\pm}, H_{W_\mp})$ vanish by bilinerarity and because of the examples of non selfdual Einstein metrics, like the Euclidean Schwarzschild metric or the product metric $\mathbb{S}^2\times \mathbb{S}^2$.
    \end{itemize}
\end{proof}
\begin{rem}
    One can also use the expression of the quadratic terms of the Ricci curvature directly from \eqref{Ric2 si scal cst}. The computations are not straightforward but the different terms remarkably cancel out as expected.
\end{rem}

\section{Function spaces and analyticity on ALE spaces}\label{appendix norm analyticity}

In this appendix, we define function spaces from \cite{ozuthese} and use it to show the analytic dependence of Einstein modulo obstructions deformations of Ricci-flat ALE metrics needed in the last section of the article. The proofs in the compact situation can be found in \cite{koi} and \cite[Chapter 1, Section 3.1]{ozuthese}.

\subsection{Function spaces}\label{sec fct space}

For a tensor $s$, a point $x$, $\alpha>0$ and a Riemannian manifold $(M,g)$. The Hölder seminorm is defined as
$$ [s]_{C^\alpha(g)}(x):= \sup_{\{y\in T_xM,|y|< \textup{inj}_g(x)\}} \Big| \frac{s(x)-s(\exp^g_x(y))}{|y|^\alpha} \Big|_g.$$

For ALE manifolds, we will consider a norm which is bounded for tensors decaying at infinity. Denote $r$ a smooth positive function equal to the parameter $d_\mathbf{e}(0,.)$ in a neighborhood of infinity where $(N,\mathbf{g}_b)$ has ALE coordinates.

\begin{defn}[Weighted Hölder norms on an ALE manifold]\label{norme ALE}
Let $\beta\in \mathbb{R}$, $k\in\mathbb{N}$, $0<\alpha<1$ and $(N,g_b)$ be an ALE manifold. Then, for any tensor $s$ on $N$, we define
   \begin{align*}
       \| s \|_{C^{k,\alpha}_{\beta}}:= \sup_{N}(1+r)^\beta\Big( \sum_{i=0}^k(1+r)^{i}|\nabla_{g_b}^i s|_{g_b} + (1+r)^{k+\alpha}[\nabla_{g_b}^ks]_{C^\alpha({g_b})}\Big).
   \end{align*}
\end{defn}
\begin{lem}[{\cite[Lemma 2.1]{biq1},\cite{ozuthese}}]
    Let $(N,\mathbf{g}_b)$ be an ALE orbifold. Then, for any $\beta \in (0,2)\cup (2,4)$ there exists $C>0$ such that for any $h\in C^{2,\alpha}_{\beta}$, $h\perp \mathbf{O}(\mathbf{g}_b)$, we have
    \begin{equation}
        \|h\|_{C^{2,\alpha}_{\beta}}\leqslant C \|\mathbf{\Phi}^{(1)}_{\mathbf{g}_b}h\|_{C^{\alpha}_{\beta+2}}.\label{inverse ALE term 4}
    \end{equation}
\end{lem}

\subsection{Real-analytic dependence of Einstein modulo obstructions metrics}\label{sec dep analytic}

Let us consider $\beta\in(0,2)$ to ensure that the kernel \emph{and} cokernel of the linearization of $\mathbf{\Phi}_{\mathbf{g}_b}$ is reduced to $\mathbf{O}(\mathbf{g}_b)$. The map $g\in C^{2,\alpha}_{\beta}\mapsto \mathbf{\Phi}_{\mathbf{g}_b}(g) \in C^{\alpha}_{\beta+2}$ is a real-analytic map between Banach spaces because the ``weights'' $(1+r)^{\beta}$ or $(1+r)^{2+\beta}$ in both the starting and target spaces are larger than $1$. This condition on the weight ensures that multilinear operations are continuous in this topology, see the theory of \cite{pal} for the source of this requirement and \cite[Proof of Lemma 8.2]{biq1} for a discussion of the weight larger than $1$ in the case of weighted Hölder norms.

We can therefore apply the implicit function theorem to the following analytic map:  $\mathbf{\Psi}:C^{2,\alpha}_{\beta}\times \mathbf{O}(\mathbf{g}_b)\times \mathbf{O}(\mathbf{g}_b) \mapsto  C^{\alpha}_{\beta+2}\times C^{2,\alpha}_{\beta}$ defined by
$$ \mathbf{\Psi}(g,\mathbf{o},v)\mapsto \Big(\mathbf{\Phi}_{\mathbf{g}_b}(g)+\mathbf{o}, \pi_{\mathbf{O}(\mathbf{g}_b)}(g-\mathbf{g}_b-v)\Big), $$
where $\pi_{\mathbf{O}(\mathbf{g}_b)}$ is the $L^2(\mathbf{g}_b)$-projection on $ \mathbf{O}(\mathbf{g}_b)$ which is linear (hence real-analytic).

The map $\mathbf{\Psi}$ is real-analytic between Banach spaces and it satisfies the assumptions of the implicit function theorem for real-analytic maps between Banach spaces of \cite{whi} about $\mathbf{g}_b$. Namely, it satisfies:
\begin{enumerate}
    \item $\mathbf{\Psi}(\mathbf{g}_b,0,0) =0$, and
    \item the linearization $(g,\mathbf{o})\mapsto \mathbf{\Psi}(g,\mathbf{o}, 0)$ is a homeomorphism by construction.
\end{enumerate}
We conclude that for any $v\in \mathbf{O}(\mathbf{g}_b)$ small enough, there exists a unique $(\bar{g}_v, \bar{\mathbf{o}}_v)\in (\mathbf{g}_b+C^{2,\alpha}_{\beta})\times \mathbf{O}(\mathbf{g}_b)$ satisfying:
$$ \mathbf{\Psi}(\bar{g}_v, \bar{\mathbf{o}}_v,v)=0 $$
and that $v\mapsto (\bar{g}_v, \bar{\mathbf{o}}_v)$ is real-analytic.

\Addresses

\end{document}